\documentclass[12pt,reqno]{amsart}
\usepackage{amsfonts}
\usepackage{lineno}
\usepackage{graphicx,subfigure}
\usepackage{amssymb}
\usepackage[alphabetic]{amsrefs}
\usepackage{epstopdf}
\usepackage{geometry}
\usepackage{float}
\usepackage{amsmath}
\usepackage[headings]{fullpage}
\usepackage{enumerate}
\usepackage{color}
\theoremstyle{plain}
\newtheorem{thm}{Theorem}[section]
\newtheorem{lem}[thm]{Lemma}
\newtheorem{cor}[thm]{Corollary}

\theoremstyle{definition}

\newtheorem{remark}[thm]{Remark}
\definecolor{darkpastelgreen}{rgb}{0.01, 0.75, 0.24}

\begin{document}
\dedicatory{Dedicated to Sibe Marde\v{s}i\'{c} (1927-2016)}
\title{New Techniques for Computing Geometric Index}

\author[K. B. Andrist]{Kathryn B. Andrist}
\address{Mathematics Department, Utah Valley University,
Orem, Utah 84058}
\email{kathy.andrist@uvu.edu }
\urladdr{http://www.uvu.edu/profpages/profiles/show/user\_id/146}
\author[D. J. Garity]{Dennis J. Garity}
\address{Mathematics Department, Oregon State University,
Corvallis, OR 97331, U.S.A.}
\email{garity@math.oregonstate.edu}
\urladdr{http://www.math.oregonstate.edu/\symbol{126}garity}
\author[D. D. Repov\v{s}]{Du\v{s}an D. Repov\v{s}}
\address{%
Faculty of Education,
and Faculty of Mathematics and Physics,
University of Ljubljana\\
Ljubljana,  SI-1000, Slovenia}
\email{dusan.repovs@guest.arnes.si}
\urladdr{http://www.fmf.uni-lj.si/\symbol{126}repovs}
\author[D. G. Wright]{David G. Wright}
\address{Department of Mathematics, Brigham Young University,
Provo, UT 84602, U.S.A.}
\email{wright@math.byu.edu}
\urladdr{http://www.math.byu.edu/\symbol{126}wright}

\date{\today}



\begin{abstract}
We introduce new {general}  techniques for computing the geometric index of a link $L$ in the interior of a solid torus $T$. These techniques simplify and unify previous ad hoc methods used to compute the geometric index in specific examples { and allow a simple computation of geometric index for new examples where the index was not previously known}. The geometric index measures the minimum number of times any meridional disc of $T$ must  intersect $L$. It is related to the algebraic index in the sense that adding up signed intersections of an interior simple closed curve $C$  in $T$ with a meridional disc gives $\pm$ the algebraic index of $C$  in $T$. One key idea is introducing the notion of geometric index for solid chambers of the form $B^2\times I$ in $T$. We prove that if a solid torus can be divided into solid chambers by meridional discs in a specific {(and often easy to obtain)} way, then the geometric index can be easily computed. \end{abstract}
\maketitle

\section{Introduction}\label{Intro Section}
The geometric index of a link $L$ in the interior of a solid torus $T$ measures the minimum number of times any meridional disc of $T$ must  intersect $L$.  The geometric index of $L$ in $T$ often seems intuitively obvious, but it is  surprisingly difficult to prove that intuition corresponds to the actual index. 
See for example the links in Figures {\ref{Complicated Links}, \ref{Index2Fig}, and \ref{Additional Links}}. Various ad hoc methods were used to compute the geometric index of the links in {Figure \ref{Additional Links} }in \cite{AW00} and \cite{GRW17}, but these methods {do} not apply to links in general. The method described in this paper encompasses the ad hoc arguments used in those papers and applies {more generally} to links in a solid torus.

In this paper we introduce new techniques for computing the geometric index of $L$ in  $T$. The strategy is to break up  the solid torus into chambers $C_i$ of the form $B^2\times I$ by a carefully selected collection of meridional discs, and to analyze the part of the link $L$ in each $C_i$. The chambers can in practice be chosen so that all but two of the segments in $C_i$ meet each end of $C_i$ in a single point.
{This can be done by isolating turning points or pairs of turning points in components of  $L$ in separate chambers. This method holds for the links in the literature that we have examined.} Each meridional disc of $C_i$ must intersect such a segment. This reduces the analysis in each such chamber to determining whether an arc with both endpoints in one end of the chamber links an arc with both endpoints in the other end. To make this precise, we define the geometric index of $L\cap C_i$ in $C_i$. If each such geometric index is $n$, and if each meridional disc in the chosen collection intersects $L$ transversely in $n$ points, then we show that the geometric index of $L$ in $T$ is $n$.

This is stated precisely in Theorem \ref{Main Theorem}. The three corollaries that follow from this theorem give various methods for computing the geometric index from information about the chambers. {One might think a simpler result could be obtained by not requiring that each meridional disc in the collection intersects $L$ in exactly $n$ points, and that one could simply take the minimum geometric index of the link in the collection of chambers. The example in Figure \ref{Index02Fig} shows that this is not the case.}

\begin{figure}[ht]
\begin{center}
  \subfigure[Index 2]%
    {%
    \label{Index2}
    \includegraphics[width=.35\textwidth]{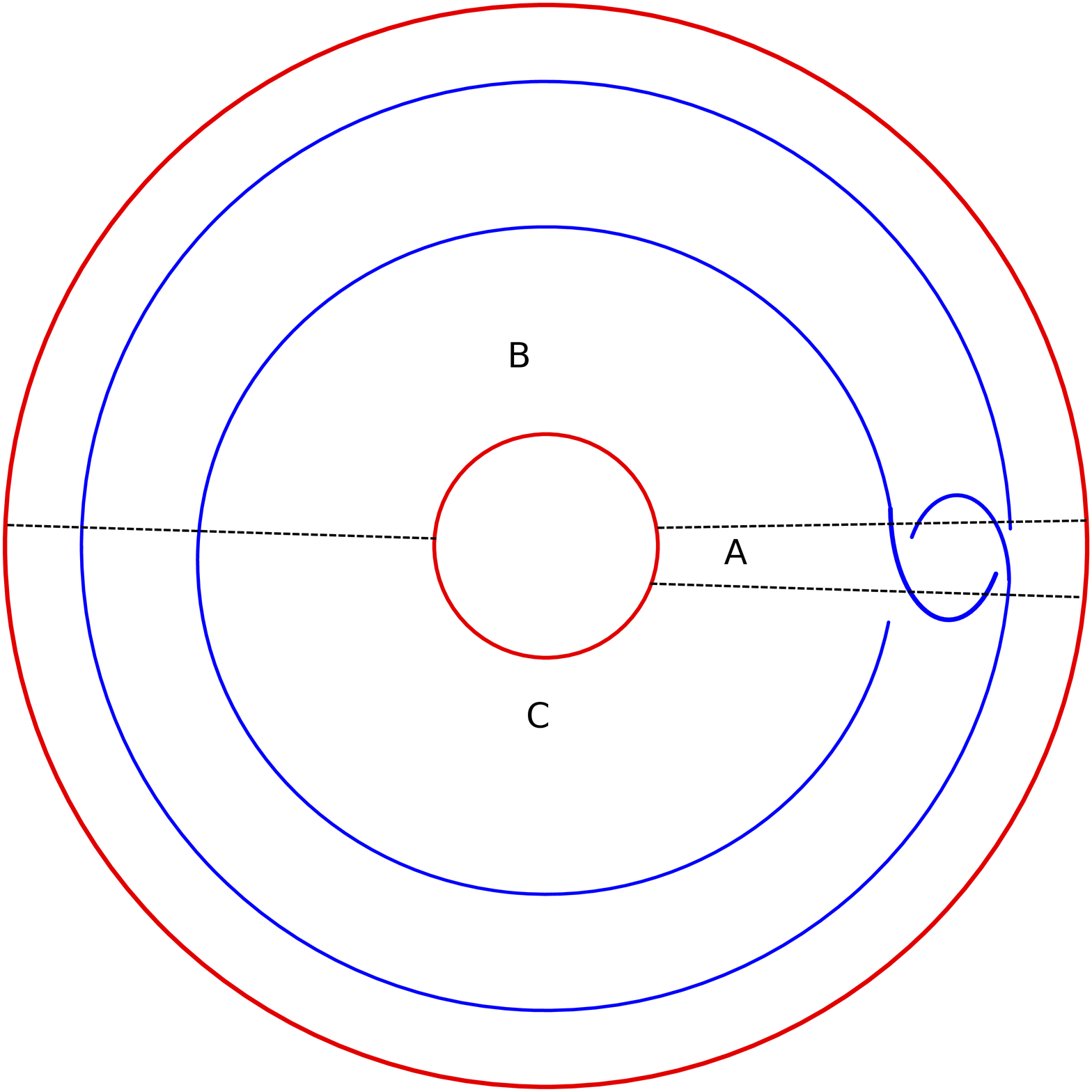}
    }%
  \subfigure[Index 0]%
    {%
    \label{Index0}
    \includegraphics[width=0.35\textwidth]{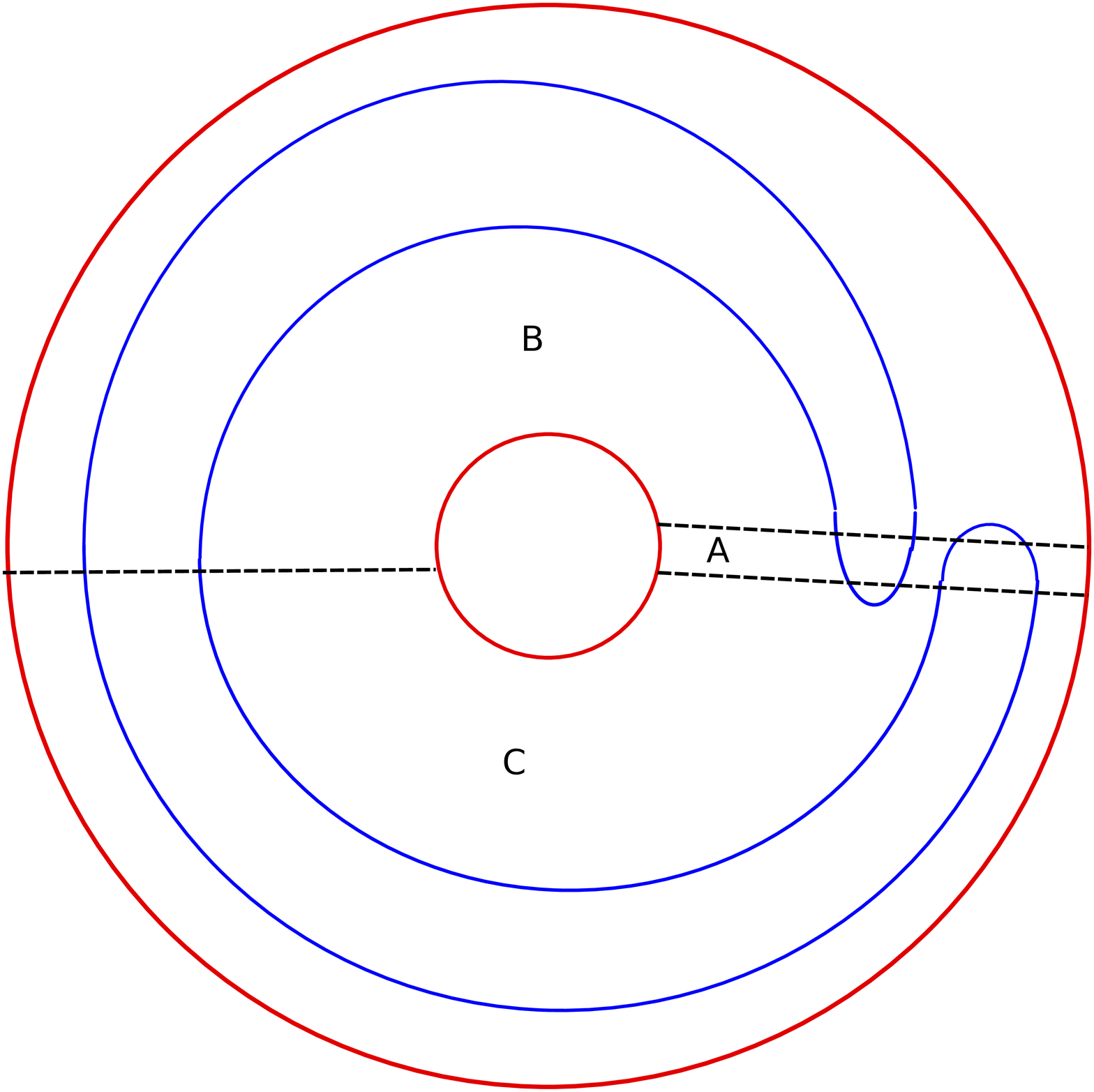}
    }  \end{center}
  \caption{%
    Chambers with Same Indices}%
  \label{Index02Fig}
\end{figure}

{In each subfigure, Chamber A has geometric index 4 and Chambers B and C have geometric index 2. But the overall geometric indices of the two examples differ.}

A {specific case} of this result for certain {special} links in the setting of only one meridional disc appears in  \cite{AW00}. Another {specific simple} case  of this result for one {specialized} type of link and two meridional discs appears in  \cite{GRW17}. 
Section \ref{Definition Section} contains terminology, definitions and basic results about geometric and algebraic indices. This section also indicates how geometric index computations can be extremely useful in proving certain geometric results.  Section \ref{Chamber Section} defines geometric index for chambers of the form $B^2\times I$ and computes the geometric index in a few special cases. Section \ref{Main Section} contains the proof of Theorem \ref{Main Theorem} and three corollaries. Section \ref{Application Section} contains {examples of} applications of the main results.

\section{Definitions and Results about Geometric and Algebraic Indices}
\label{Definition Section}

If $S$ is a solid torus embedded in another solid torus $T$, the \emph{algebraic index} of $S$ in $T$, $a(S,T)$, is defined to be $\vert \alpha \vert$, where $\alpha$ is the integer in $H_{1}(T)$ represented by the center line of $S$. Algebraic index is multiplicative, so that if
$S\subset T\subset U$ are solid tori, the algebraic index of $S$ in $U$ is the product of the algebraic index of $S$ in $T$ with the algebraic index of $T$ in $U$. Note that the algebraic index of a Whitehead link in the torus containing it is $0$, as is the algebraic index of each component of a Bing link. See Figure \ref{Index2Fig}.

Schubert introduced the notion of geometric index in \cite{Sch53}. If $K$ is a link in the interior of a solid torus $T$, then we denote the \emph {geometric index} of $K$ in $T$ by $N(K,T)$.  The geometric index is the minimum, over all meridional discs $D$ of $T$,  of $|K \cap D|$.  A \emph {core} of a solid torus $T$ in 3-space is a simple closed curve $J$ such that $T$ is a regular neighborhood of $J$.  Likewise, a core for a finite union of disjoint solid tori is a link consisting of one core of each of the solid tori.   If $T$ is a solid torus and $M$ is a finite union of disjoint solid tori so that $M \subset  Int  \ T$, then the geometric index $N( M,T)$ of $M$ in  $T$ is $N(K,T)$, where $K$ is a core of $M$.  

{Figure \ref{Complicated Links} indicates a collection of new examples. Each circled part of the inner link in Figure \ref{Complicated} can be replaced by one of the patterns in  Figure \ref{Strands} to produce many different examples. These examples do not fit into any previous class of examples such as the Gabai or McMillan links in \cite{GRW17}. More complicated links can also be obtained by additional varying of the linking or winding pattern of the strands as one progresses out from the center of the diagram.}

{It is clear that the algebraic index of the inner link  in Figure \ref{Complicated} is 0. Our techniques show that the geometric index is $8$ for each example. See Section \ref{Application Section}.}

\begin{figure}[ht]
\begin{center}
  \subfigure[Complicated Link(s)]%
    {%
    \label{Complicated}
    \includegraphics[width=0.7\textwidth]{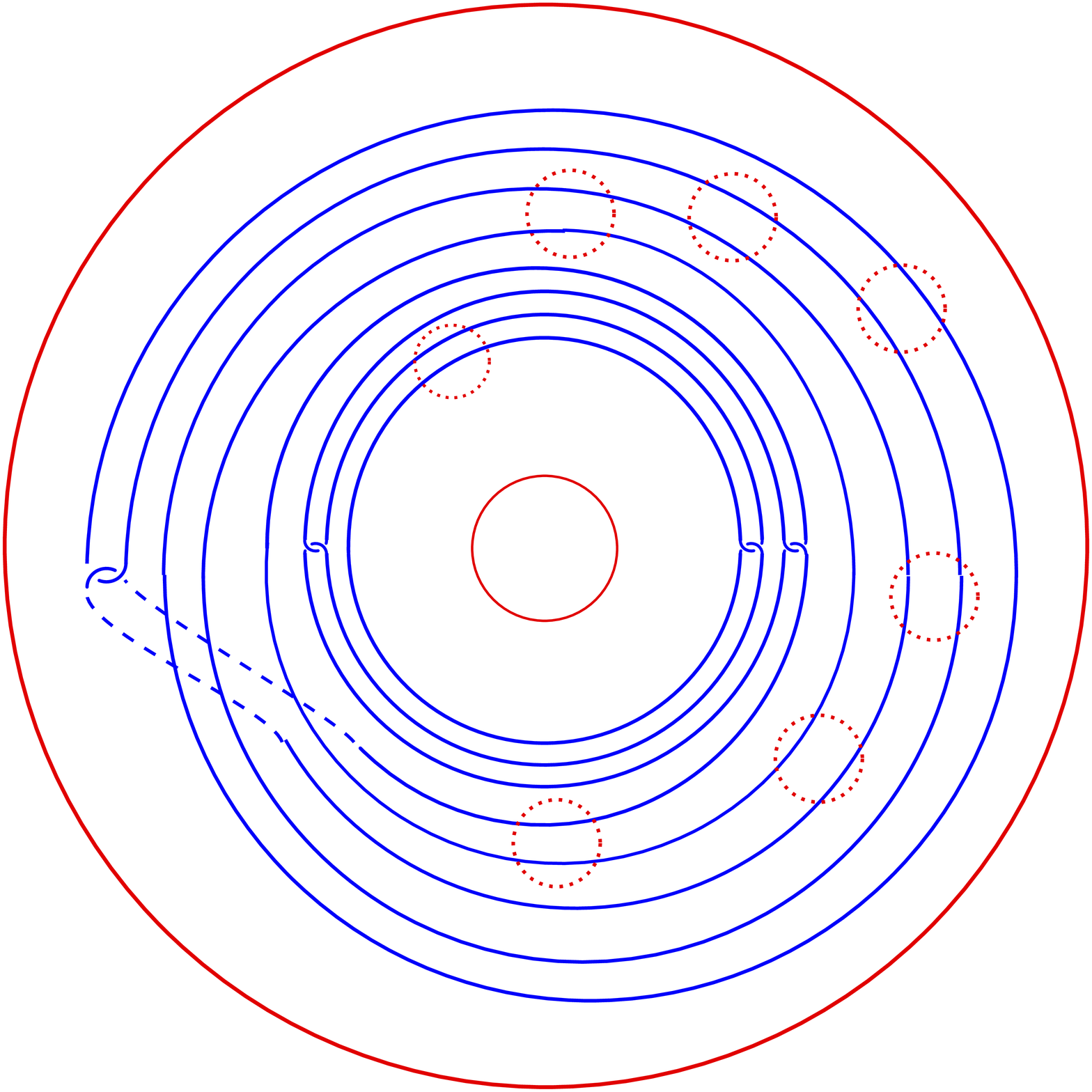}
    }\\%
  \subfigure[Replacement Strands]%
    {%
    \label{Strands}
    \includegraphics[width=0.6\textwidth]{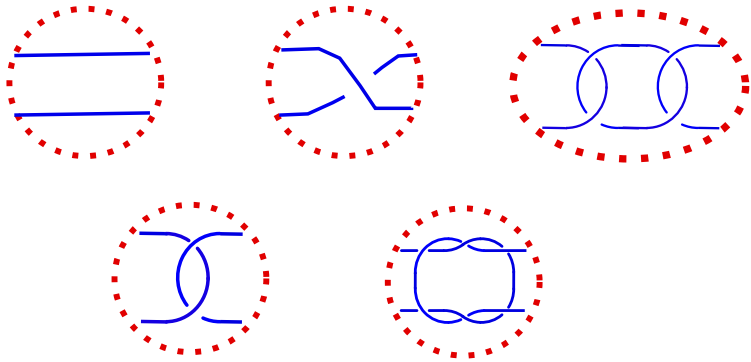}
    }%
  \end{center}
  \caption{%
    Some Complicated Links}%
  \label{Complicated Links}
\end{figure}

The geometric indices of the links in Figure \ref{Index2Fig} are all two.
The algebraic index of the Whitehead Link in Figure \ref{Whitehead} is zero as is the algebraic index of each component in Figure \ref{Bing} and Figure \ref{Antoine}. The algebraic index of the link in Figure \ref{Algebraic2} is two. The geometric indices of the links in Figure \ref{Additional Links} are indicated in the figure. More details will be provided in Section \ref{Application Section}. See \cite{GRW17} for more discussion of the Gabai and McMillan links in Figure \ref{Additional Links}.

\begin{figure}[ht]
\begin{center}
  \subfigure[Whitehead Link]%
    {%
    \label{Whitehead}
    \includegraphics[height=.4\textwidth]{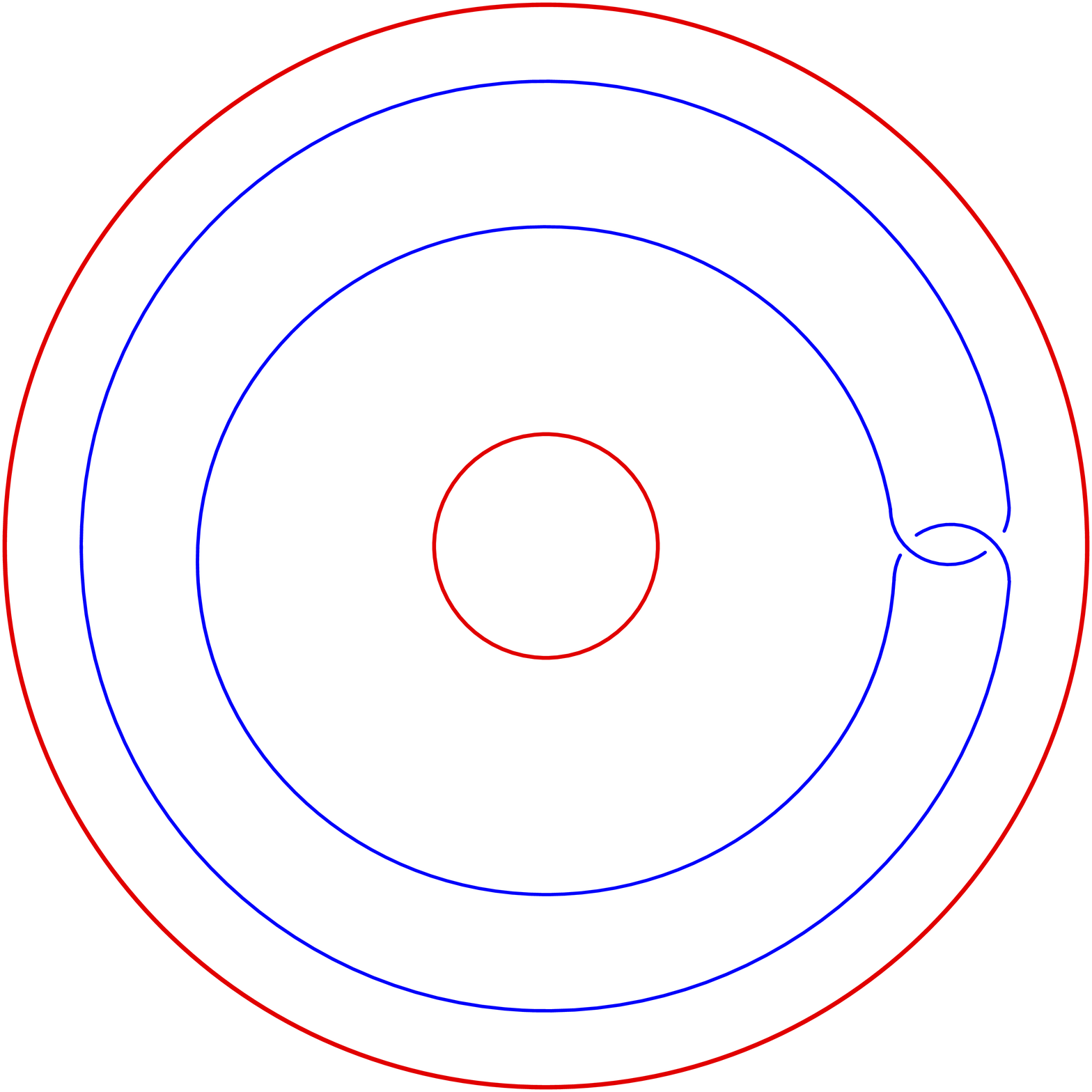}
    }%
  \subfigure[Bing Link]%
    {%
    \label{Bing}
    \includegraphics[width=0.4\textwidth]{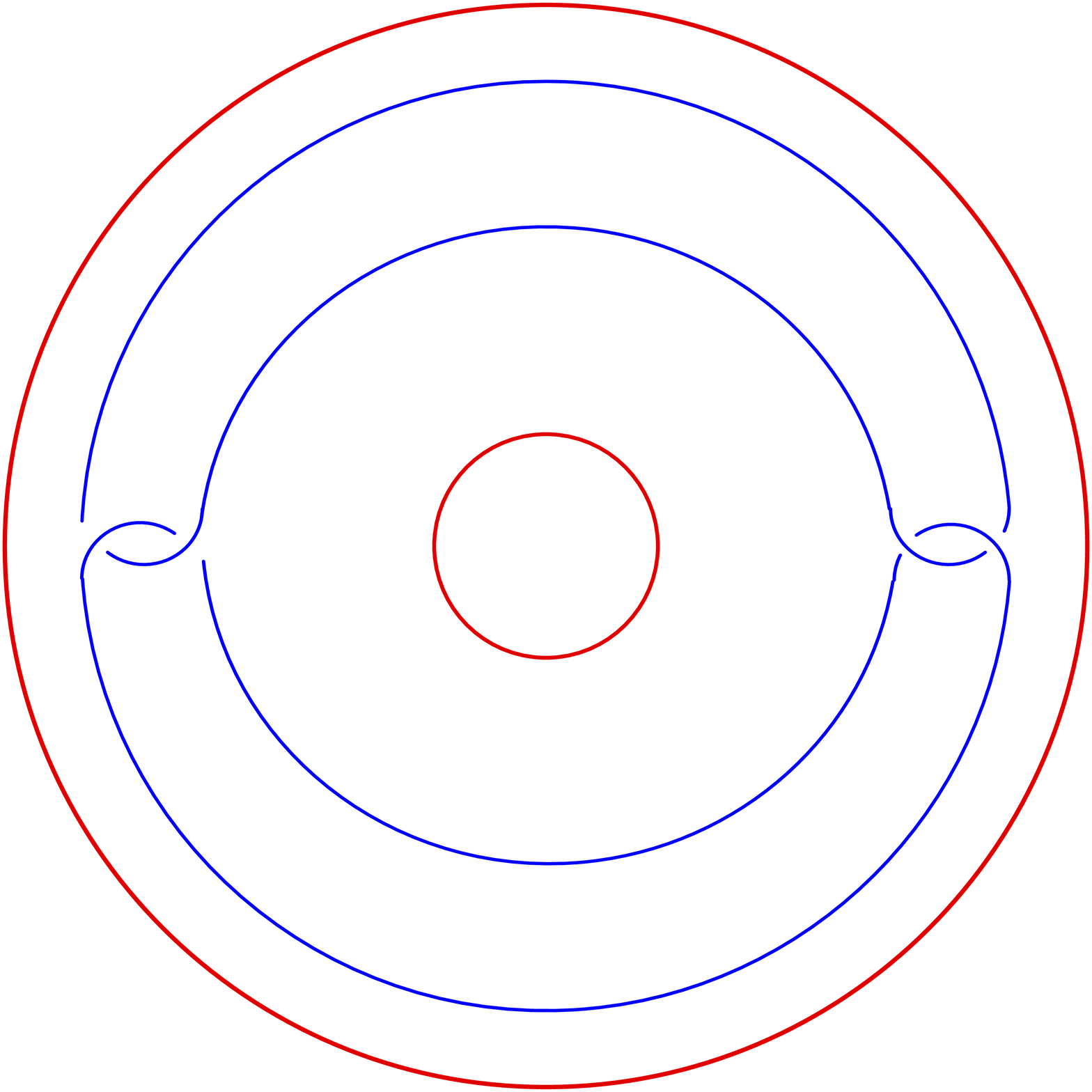}
    }\\%
  \subfigure[Antoine Link]%
    {%
    \label{Antoine}
    \includegraphics[width=0.4\textwidth]{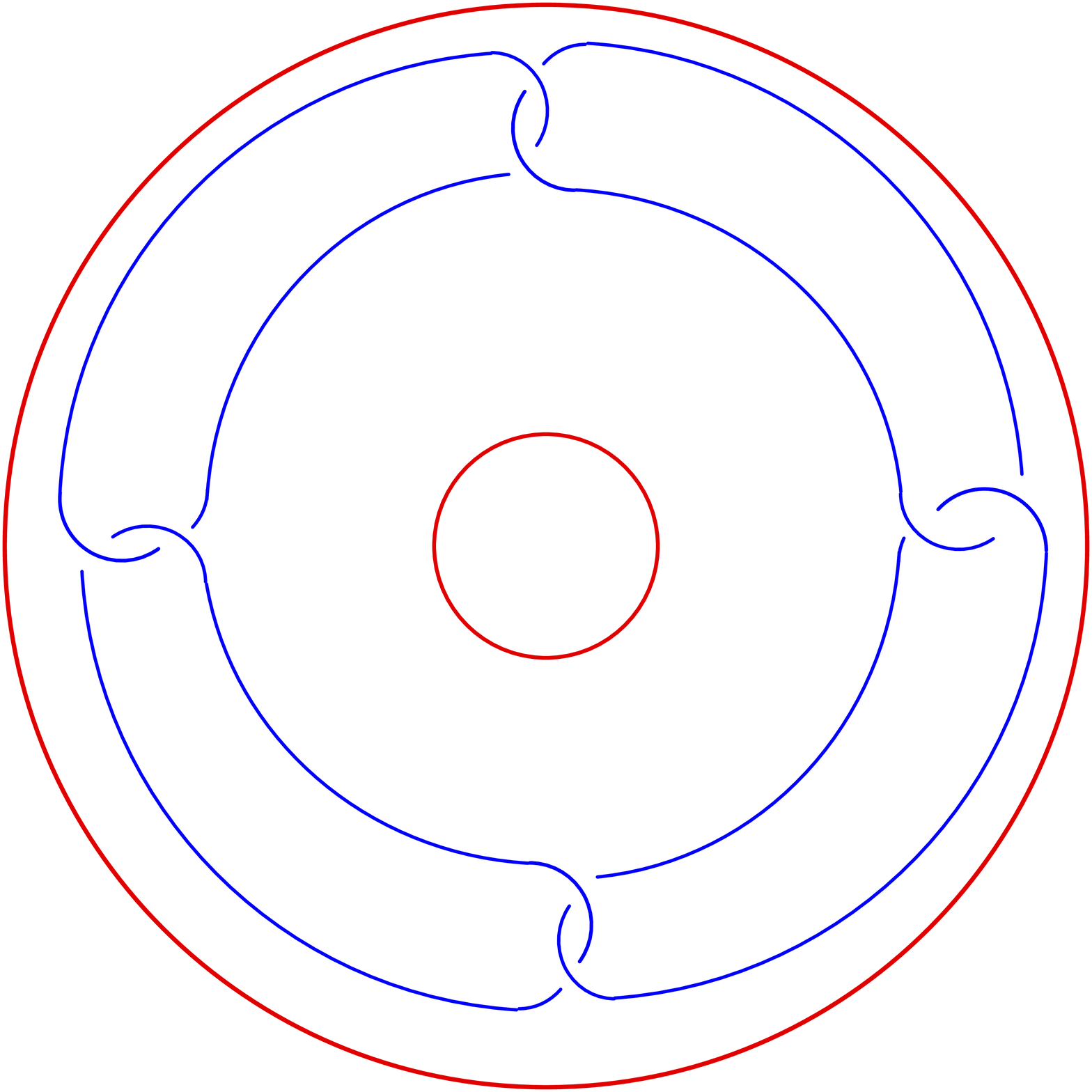}
    }%
  \subfigure[Algebraic Index 2]%
    {%
    \label{Algebraic2}
    \includegraphics[width=0.4\textwidth]{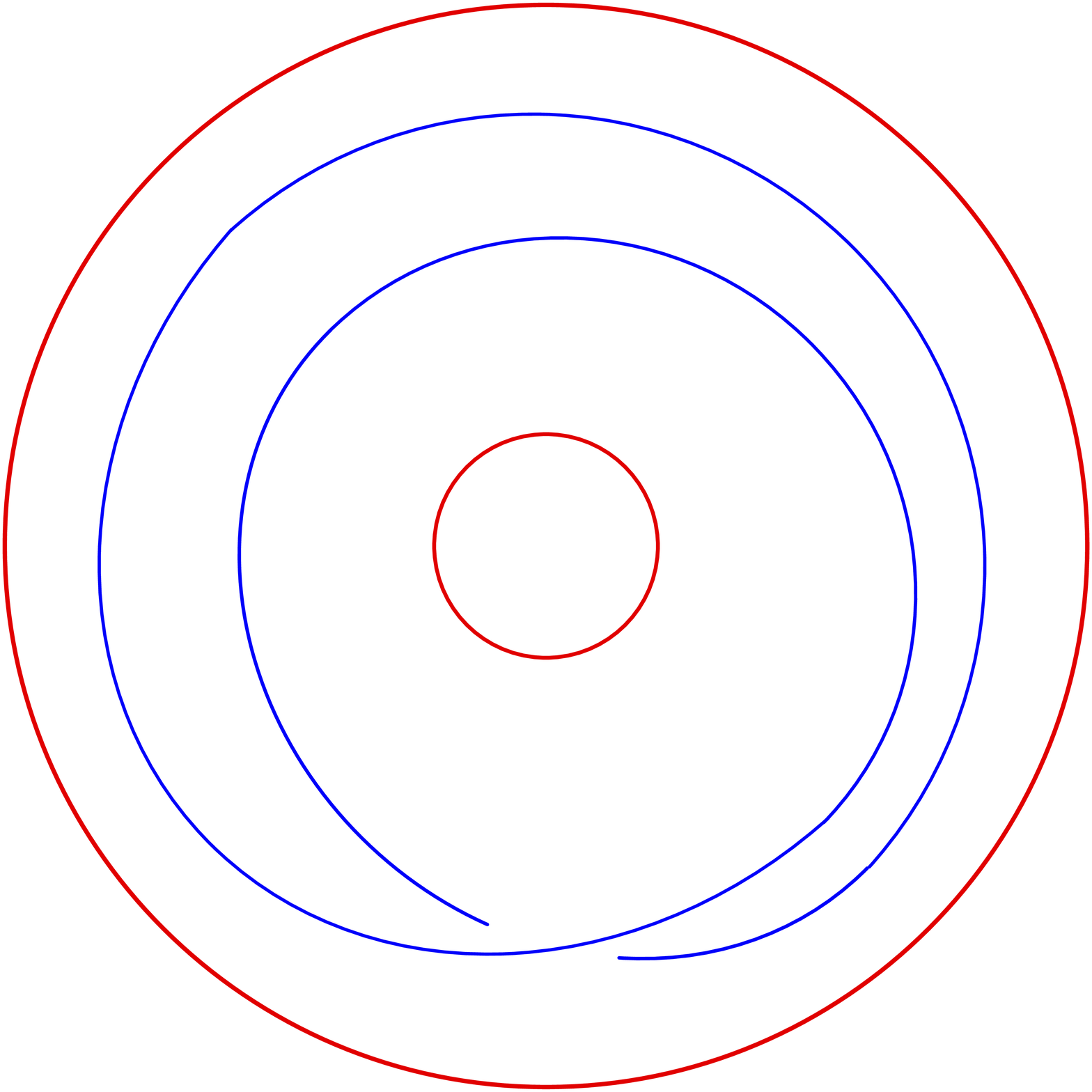}
    }%
  \end{center}
  \caption{%
    Links with Geometric Index 2}%
  \label{Index2Fig}
\end{figure}

\begin{remark}\label{algebraic computation}
It is well known that the algebraic index of a simple closed curve $J$ in a solid torus $T$ can be computed as follows. Choose a meridional disc $D$ that intersects $J$ transversely. Assign a positive and a negative side to $D$ in $T$ so that a simple closed curve passing through $D$ once, going from the negative to the positive side, has algebraic index $1$ in $T$. Assign a $\pm 1$ to each intersection of $J$ with $D$ corresponding to passing from the negative to the positive or from the positive to the negative side. Then the algebraic index of $J$ in $T$ is the absolute value of the sum of the numbers assigned to each intersection of $J$ with $D$. {In particular, the algebraic and geometric indices have the same parity.}
\end{remark}

\begin{remark} Note that the previous remark immediately implies that the geometric index is greater than or equal to the algebraic index.
\end{remark}

\begin{figure}[h]
\begin{center}
  \subfigure[Knotted Link - Index 3]%
    {%
    \label{Knotted}
    \includegraphics[width=0.33\textwidth]{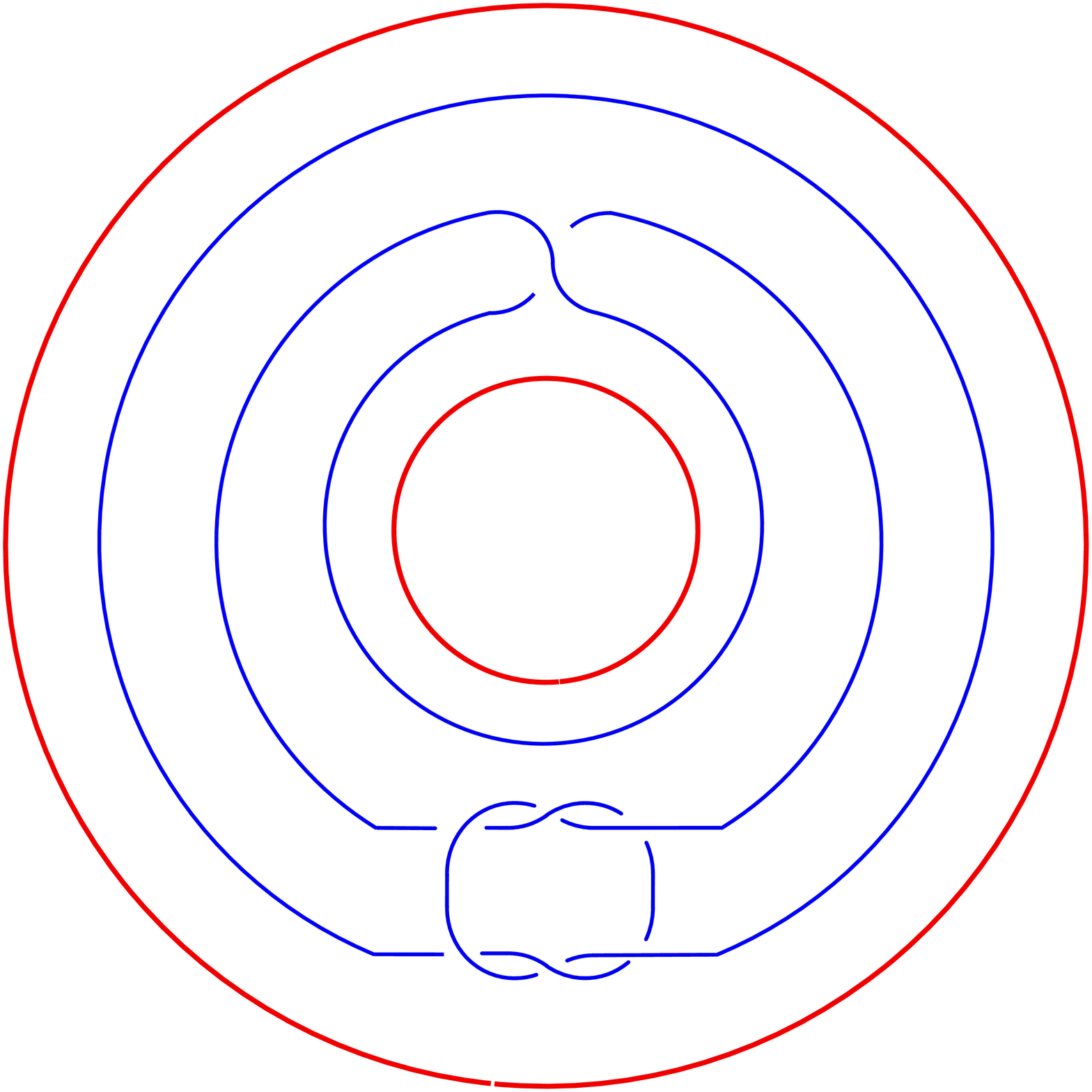}
    }
  \subfigure[McMillan Link - Index 4]%
    {%
    \label{McM4}
    \includegraphics[width=0.33\textwidth]{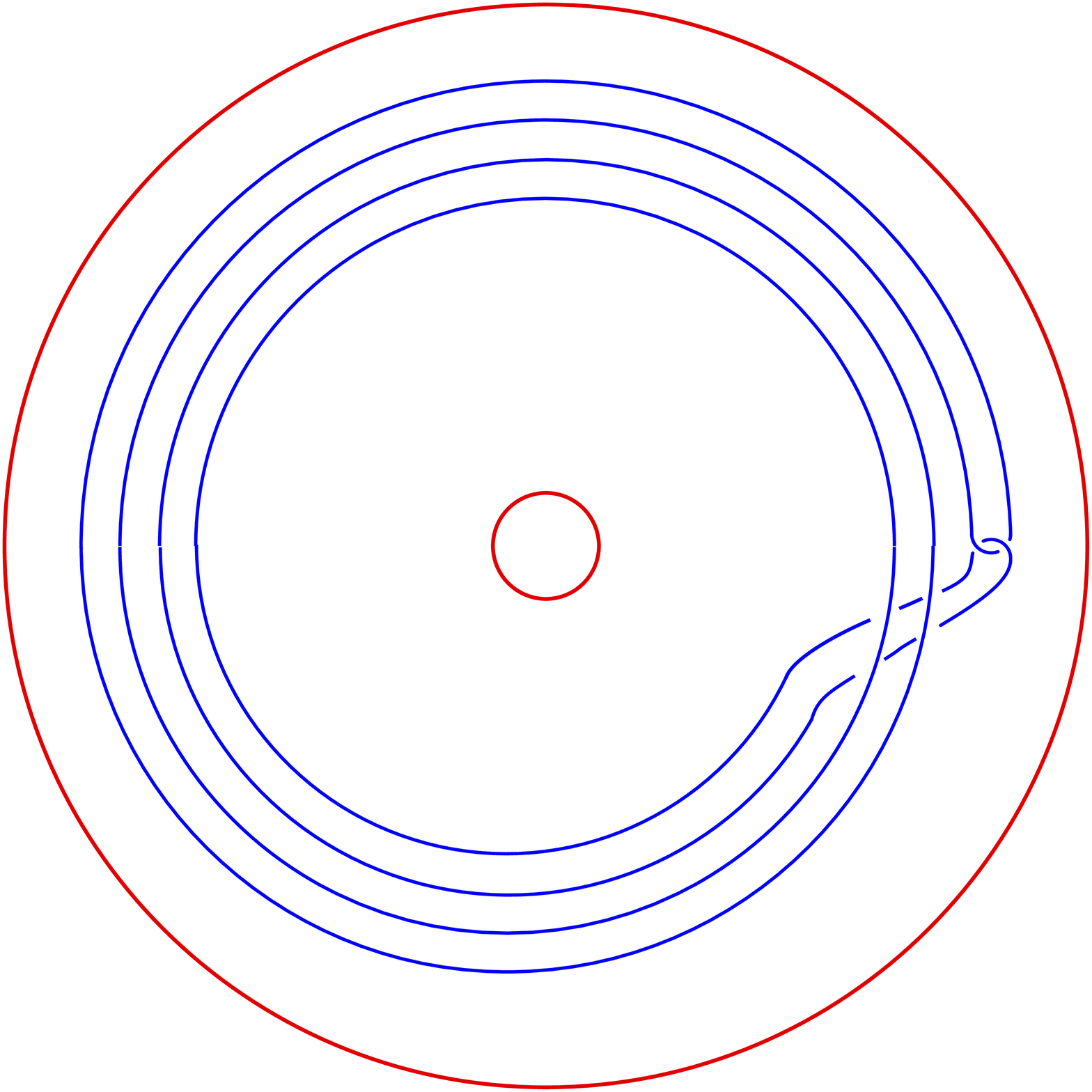}
    }%
  \subfigure[Gabai Link - Index 6]%
    {%
    \label{Gabai6}
    \includegraphics[width=0.33\textwidth]{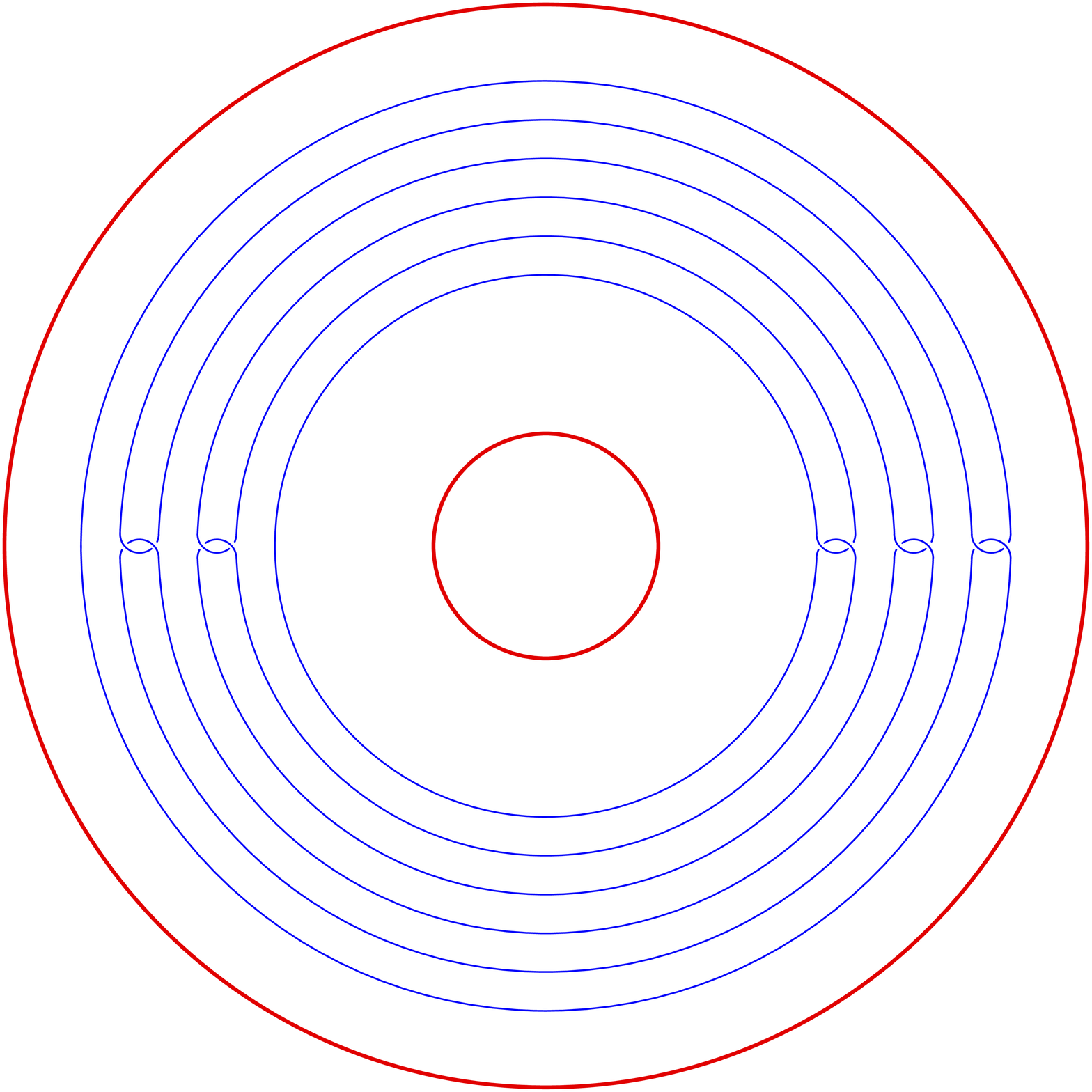}
    }%
  \end{center}
  \caption{%
    Additional Links}%
  \label{Additional Links}
\end{figure}

The next few results follow from Schubert \cite{Sch53} and regular neighborhood theory \cite{RS82}. They indicate the usefulness of being able to compute geometric index. For more details on the proofs of these lemmas, see Section 3 in \cite{GRWZ11}. We include a proof of the third lemma to give an indication of the techniques used.

\begin{lem}\label{Index 1 Lemma}  Let $T_0$ and $T_1$ be unknotted solid tori in $S^{3}$ with  $T_0 \subset Int T_1$ and $N( T_0, T_1) = 1$.  Then $\partial T_0$ and  $\partial T_1$ are parallel; i.e., the manifold $T_1 -  Int  T_0$ is homeomorphic to $\partial T_0 \times I$ where $I$ is the closed unit interval $[0,1]$.
\end{lem}

\begin{lem}\label{Product Lemma}  Let $T_0$ be a finite union of disjoint solid tori in $S^3$. Let $T_1$ and $T_2$ be solid tori so that $T_0 \subset Int  T_1$ and $T_1 \subset Int  T_2$.  Then $N(T_0, T_2) =  N(T_0, T_1) \cdot  N(T_1, T_2)$.
\end{lem}

\begin{lem}\label{Even Index Lemma}Let $T$ be a solid torus in $S^{3}$ and let $T_{1},\ldots T_{k}$ be disjoint unknotted solid tori in $T$, each of geometric index $0$ in $T$. Then the geometric index of $\cup_{i=1}^{k}T_{i}$ in $T$ is even.
\end{lem}

\begin{proof}
If the geometric index were odd, then there would be a meridional disc $D$ of $T$ that would intersect the cores of  $\cup_{i=1}^{k} T_{i} $ transversally an odd number of times.  This would mean that $D$ must intersect the core of some $T_{i}$  an odd number of times.  But if a meridional disc of $T$ intersects a simple closed curve $J$ transversally an odd number of times, the algebraic index of $J$ in $T$ is odd by Remark \ref{algebraic computation}, and so $J$ is essential in $T$.  However, for each $i$, the core of  $T_{i}$ is inessential because it lies in a ball in $T$. 
\end{proof}

The next two results make use of the material on geometric index to determine when the boundaries of certain tori are parallel as a further illustration of the use of geometric index. See \cite{Wri89} for details. We include the proof of the first theorem to further indicate the techniques used.

\begin{thm}\label{Whitehead Parallel Theorem}  {Let $W$ be torus in a solid torus $T$ in  $S^3$ with the core of $W$  forming a Whitehead link (as in Figure \ref{Whitehead}) in  $T$}.  If $\ T' \subset T$ is a solid unknotted torus whose boundary separates $\partial W$ from $\partial T$, then $\partial T'$ is parallel to either $\partial W$ or $\partial T$.
\end{thm}
\proof Since $\partial T'$ separates  $\partial W$ from $\partial T$, we have $W \subset Int T'$ and $T' \subset Int T$.  Since $N(W,T') \cdot N(T',T)=N(W,T)=2$, either $N(W,T')=1$ or $N(T',T)=1$.  The conclusion now follows from Lemma \ref{Index 1 Lemma}.\qed

\begin{thm}\label{Bing Parallel Theorem}
{Let $F_1 \cup F_2$ be a pair of disjoint tori in a solid torus $T$ in $S^3$, with the cores of $F_1$ and  $F_2$ forming a Bing link (as in Figure \ref{Bing}) in  $T$}.  If $S$ is the boundary of a solid unknotted torus that separates $\partial F_1 \cup \partial F_2  \cup \partial T$, then $S$ is parallel  to one of $\partial F _1$, $\partial F _2$, $\partial T$.
\end{thm}

\section{Geometric Index for Chambers}\label{Chamber Section}

In this section, we introduce the concept of geometric index for chambers of the form $B^2\times I$, 
formed by pairwise disjoint meridional discs in a solid torus.

A \emph{chamber} $C$ is a space homeomorphic to  $B^2\times I$. The \emph{top} of the chamber, $C_t$, corresponds to  $B^2\times \{1\}$ and the \emph{bottom} of the chamber, $C_b$, corresponds to  $B^2\times \{0\}$. We think of the chamber $C$ as the region between two meridional discs  in a solid torus, with the meridional discs corresponding to $C_t\text{ and }C_b$. An \emph{interior meridional disc} $M$ in such a chamber $C$ is a disc in  $C-(C_t\cup C_b)$,  where $\partial M$ is essential in $\partial C-(C_t\cup C_b)$.
Let $L$ be a collection of arcs and simple closed curves in a chamber $C$ so that each arc has its endpoints in $C_t\cup C_b$ and so that each simple closed curve is in $C-\partial C$. The \emph{geometric index} of $L$  in $C$, $N(L,C)$, is the minimum  of $\vert L\cap M\vert $ over all interior meridional discs $M$ of $C$.

\begin{figure}[h]
\begin{center}
  \subfigure[Link in Chamber]%
    {%
    \label{ChamberLink}
    \includegraphics[width=.2\textwidth]{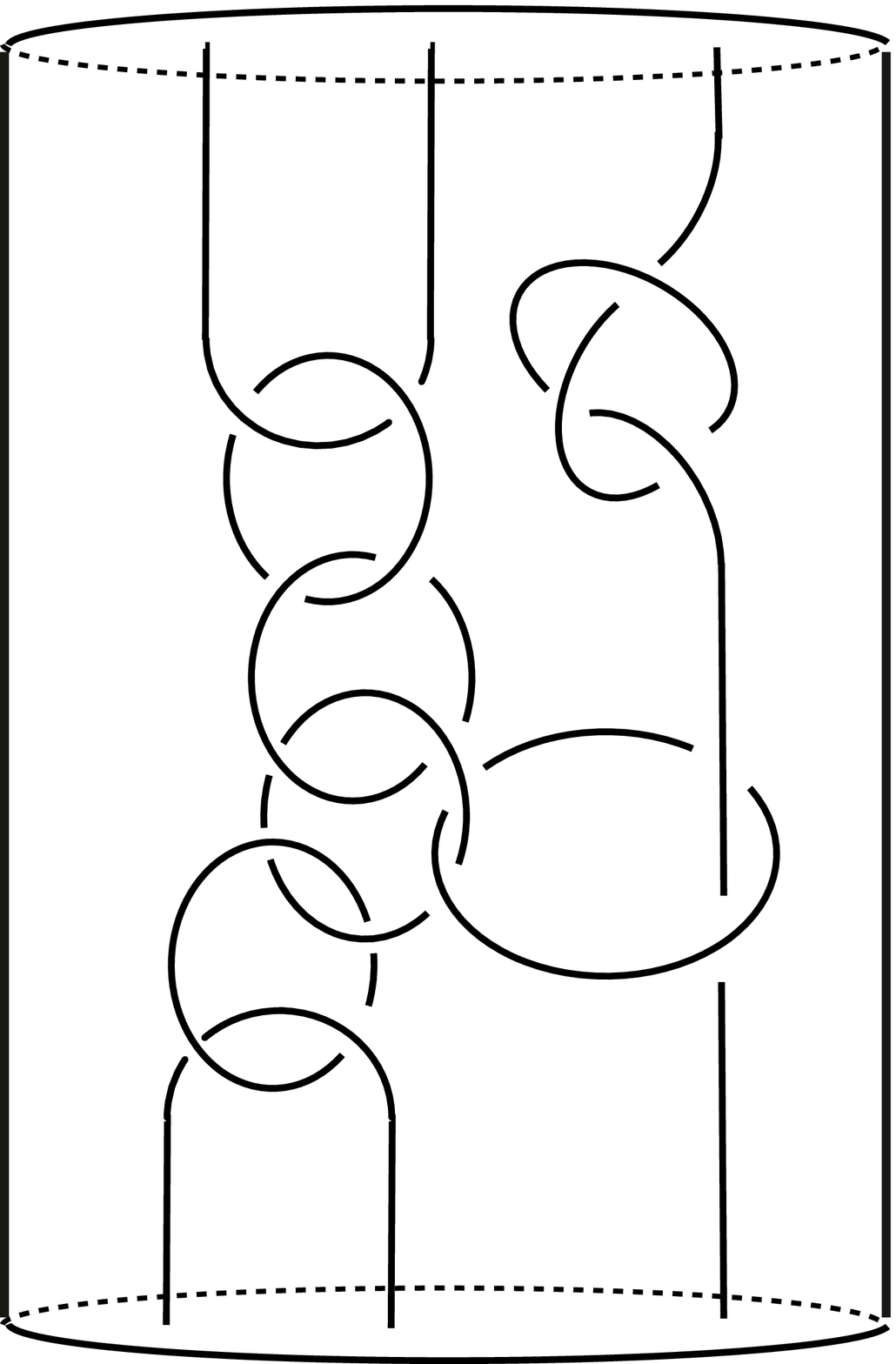}
    }\hspace{1em}
   \subfigure[Whitehead Clasp]%
    {%
    \label{ClaspWhitehead}
    \includegraphics[width=.2\textwidth]{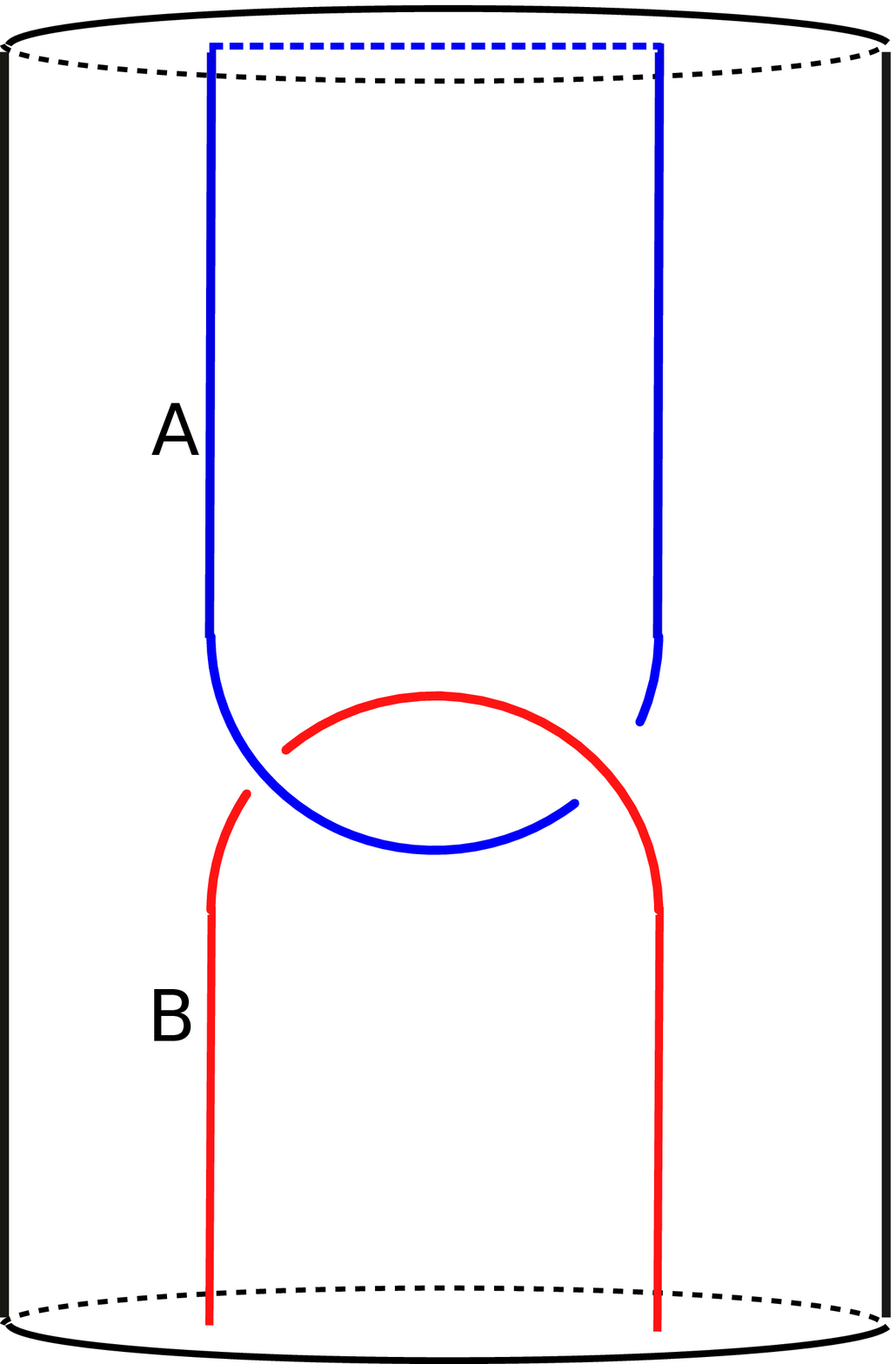}
    }\hspace{1em}%
  \subfigure[Square Knot Clasp]%
    {%
    \label{ClaspSquare}
    \includegraphics[width=.2\textwidth]{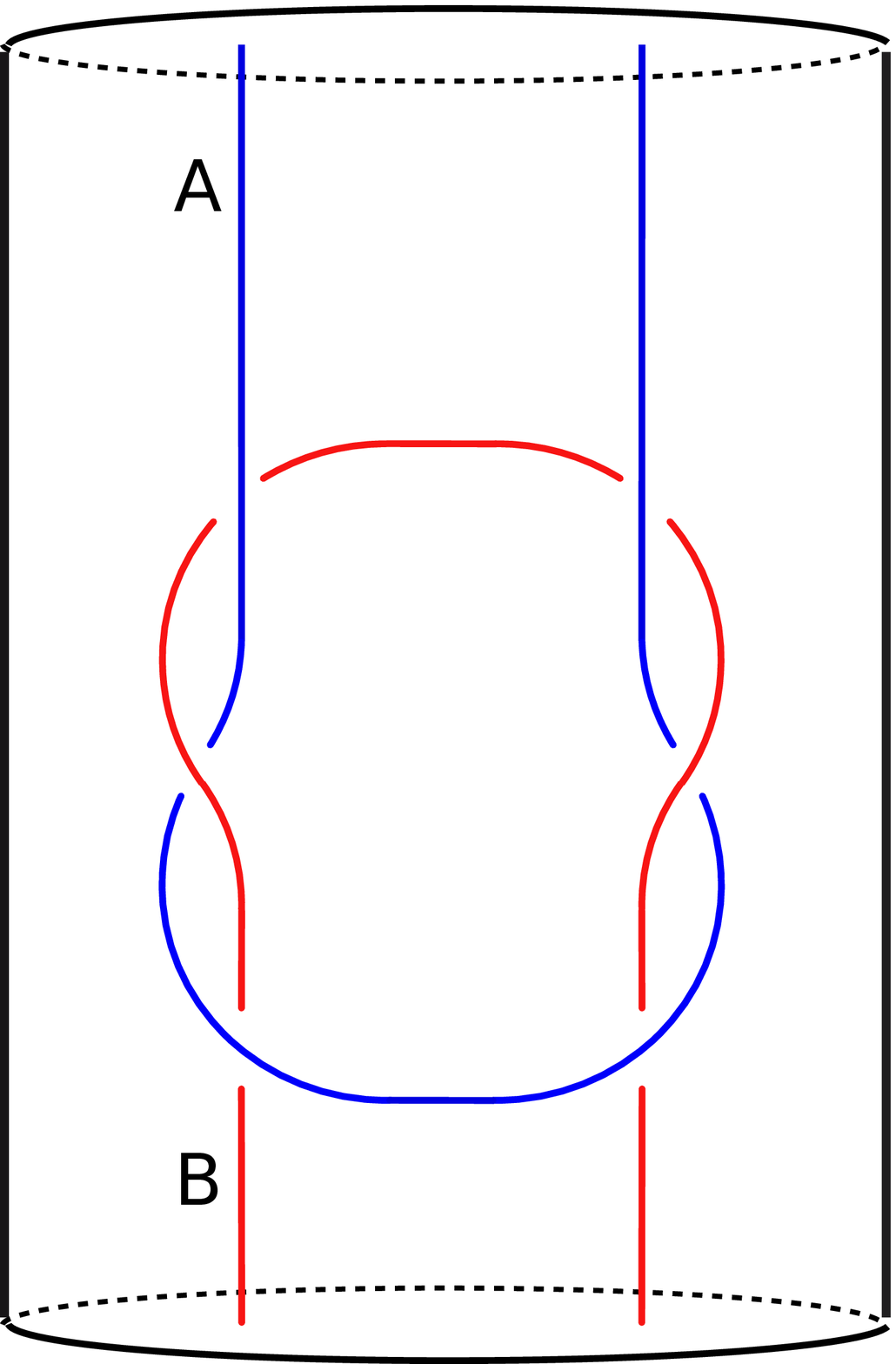}
    }\hspace{1em}
    \subfigure[Spanning Arcs]%
    {%
    \label{SpanningArcs}
    \includegraphics[width=.2\textwidth]{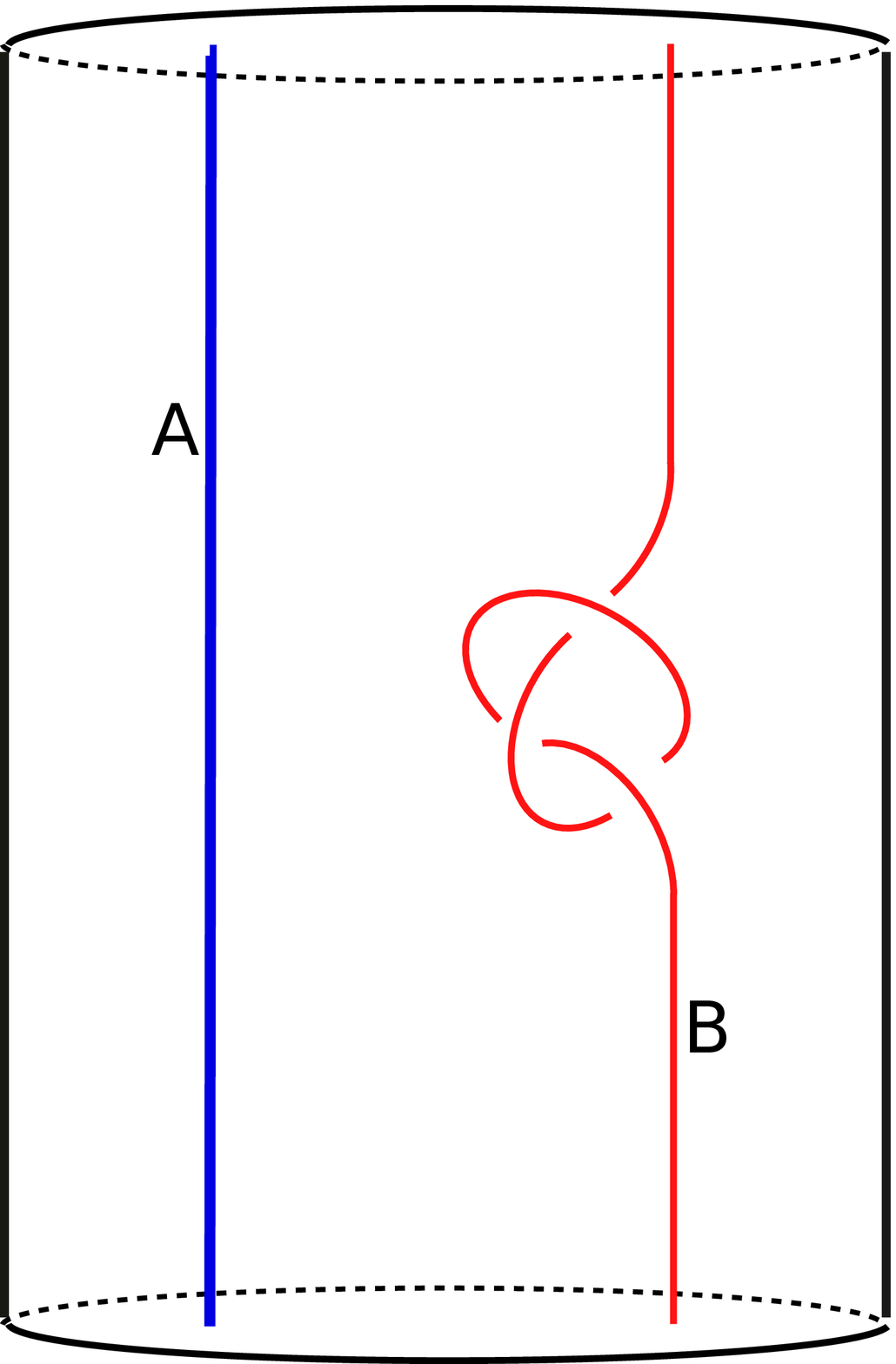}
    }
    \subfigure[Antoine Clasp]{%
    \label{AntoineClasp}
    \includegraphics[width=.2\textwidth]{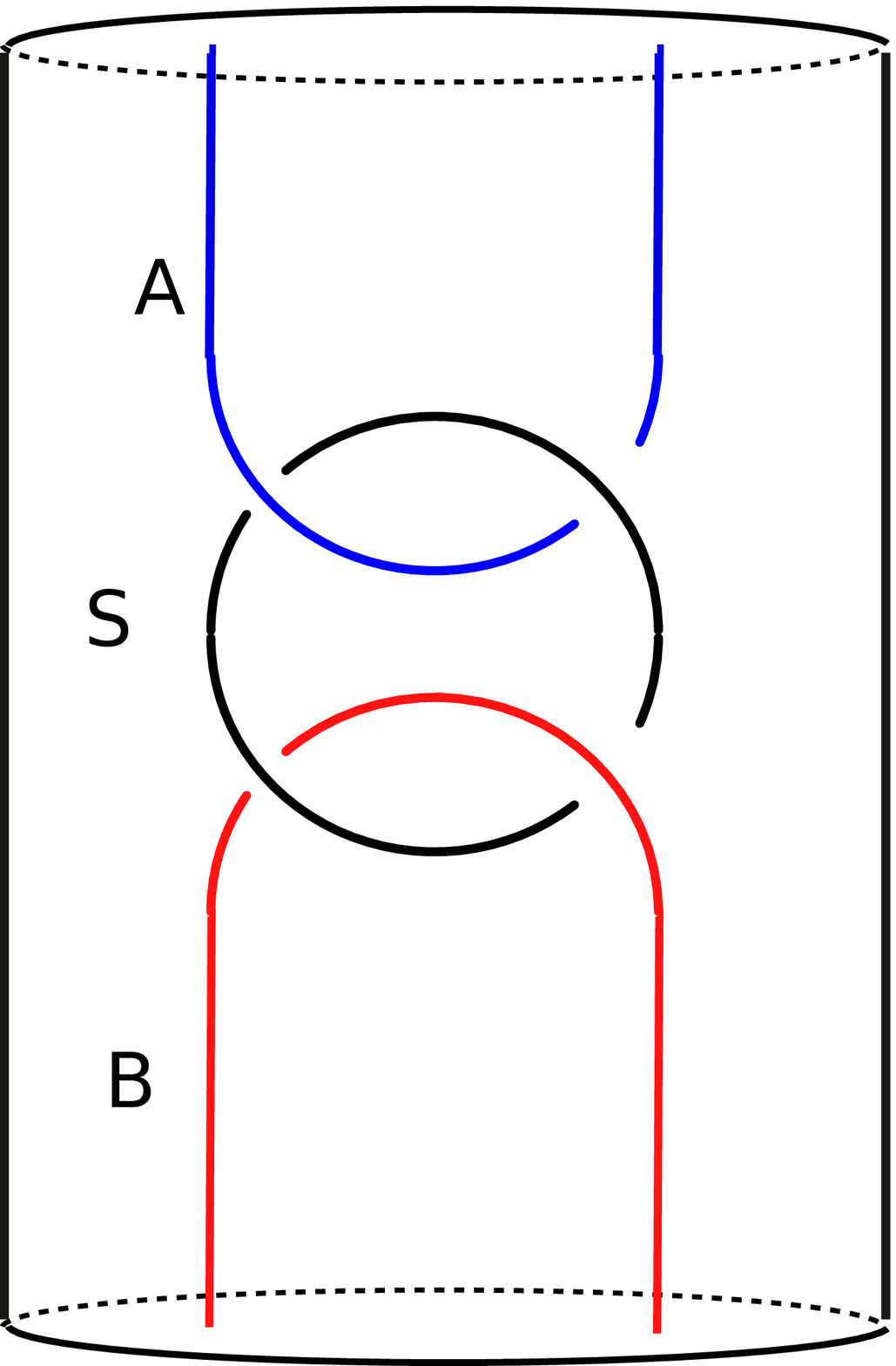}
    }
  \end{center}
  \caption{%
    Clasps and Spanning Arcs}%
  \label{Clasps-Arcs}
\end{figure}

{Consider a link  $L=A\cup B\cup S$ where $A$ is an arc with endpoints in $C_t$, $B$ is an arc with endpoints in $C_b$, and $S$ is a possibly empty union of circles. The link $L$ forms a \emph{clasp} in $C$ if the geometric index of $L$ in $C$ is two.} A \emph{spanning arc} A in $C$ is an arc with one endpoint in $C_t$, and the other endpoint in $C_b$. See Figure \ref{Clasps-Arcs} for some examples.

\begin{lem}
\label{Chamber Index Lemma}
{The Whitehead clasp, the Square Knot clasp, and the Antoine clasp (pictured in Figure \ref{ClaspWhitehead},  Figure \ref{ClaspSquare}, and Figure \ref{AntoineClasp}) have geometric index 2 in the indicated chambers.} Each spanning arc in a chamber has geometric index 1 in that chamber.
\end{lem}
\begin{proof} Any interior meridional disc must intersect a spanning arc, even if the arc is knotted in the chamber. To see this, note that any  meridional disc divides the chamber into two components  with $C_b$ in one component and $C_t$ in the other. If the meridional disc misses  the arc completely, then the arc  is contained in one of the two components which cannot happen.

\begin{figure}[ht]
\begin{center}
    \includegraphics[height=0.25\textwidth]{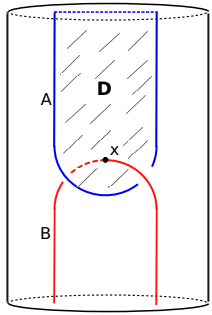}
\end{center}
\caption{Whitehead Clasp with Spanning Disc}
\label{WhiteheadClasp-Disc}\end{figure}

For the Whitehead clasp, consider Figure \ref{WhiteheadClasp-Disc}. Let $D$ be a disc in the chamber with boundary $A$ together with an arc in one end of the chamber  that joins the two points of $A$ in that end. There is no loss of generality in only considering interior meridional discs that intersect $A\cup B$ transversely  since non-transverse intersections can be removed by a small general position move which only decreases the number of intersections. By an argument similar to that in the proof of Lemma \ref{Even Index Lemma}, every such meridional disc must intersect each of $A$ and $B$ an even number of times. There is clearly an interior  meridional disc that intersects $A\cup B$ twice, so we must only eliminate the possibility of no intersections.

Suppose $M$ is an interior meridional disc that misses $A\cup B$. By a general position adjustment away from 
$A\cup B$, we can assume that $M$ intersects the disc $D$ in a finite number of simple closed curves. Consider an innermost such curve in $M$. If this curves bounds a disc in $D$ missing $x$, the disc in $M$ bounded by the curve together with the disc in $D$ bounded by the curve bound a 3 ball. This $3$-ball can be used to push the disc in $M$ to the disc in $D$ which can then be pushed slightly off $D$.   This process reduces the number of curves of intersection of $M$ with $D$. 

If this innermost curve in $M$ bounds a disc in $D$ containing $x$, the disc in $M$ bounded by the curve together with the disc in $D$ bounded by the curve form a 2-sphere that is pierced once by $B$. However, this is impossible.  So eventually all curves of  intersection of $M$ with $D$ can be eliminated. But then $M$ misses a spanning arc of the chamber which is impossible. So no interior meridional disc $M$ can miss $A\cup B$.

{The proofs for the Square Knot clasp and the Antoine clasp are similar. Note that the Antoine clasp can be divided into two Whitehead clasps by adding a meridional disc in the middle of the cylinder in Figure \ref{AntoineClasp}. Then the techniques of Theorem \ref{Main Theorem} below can be used to show that the geometric index is two.}

\end{proof}

\section{Main Results}\label{Main Section}

We now work towards proving the main result (Theorem \ref{Main Theorem}) and some corollaries.

\textbf{Setup:} Throughout this section, let $L$ be a link in the interior of a solid torus $T$ and let $D_0, \ldots D_{m-1}$ be a cyclically ordered pairwise disjoint collection of meridional discs in $T$, each intersecting $L$ transversally in $n$ points. Let $C_i$ be the chamber bounded by $D_{i-1}$ and $D_{i \mod m}$ for $1\leq i\leq m$. 

\begin{remark}\label{At Most n Remark} Note that $N(L\cap C_i,C_i)\leq n$  for each $i$. This follows from the fact that an interior meridional disc parallel to and close enough to the ends of $C_i$ intersects $L\cap C_i$ in $n$ points by the setup.
\end{remark}

\begin{lem}\label{Missing M-Discs Lemma} If  $N(L\cap C_i,C_i)\geq n$  for each $i$, and  if K is a meridional disc for $T$ that misses $\cup_{i=1}^k D_i$, then 
$\vert L \cap K\vert \geq n$.
\end{lem}
\begin{proof}This follows since any such meridional disc is an interior meridional disc for one of the chambers.
\end{proof}

\begin{lem}\label{Replacement Disc Lemma}
Assume  $N(L\cap C_i,C_i)\geq n$  for each $i$. Let $D^{\prime}$ be a disc that lies in the interior of $T$ with $D^{\prime}\cap (\cup_{i=0}^{m-1}D_i)=\partial D^{\prime}$. Then $\partial D^{\prime}$ is in some $D_i$. Let $D^{\prime\prime}$ be the disc bounded by $\partial D^{\prime}$ in $D_i$. Then $\vert L\cap D^{\prime}\vert \geq \vert L\cap D^{\prime\prime}\vert $.
\end{lem}
\begin{proof}
Suppose, to the contrary, that $\vert L\cap D^{\prime}\vert < \vert L\cap D^{\prime\prime}\vert $. Then the meridional disc
$(D_i - D^{\prime\prime}) \cup D^{\prime}$ meets $L$ in fewer points than $D_i$. This disc can be pushed slightly off $D_i$ so that
it still meets $L$ in fewer points than $D_i$. But this contradicts Lemma \ref{Missing M-Discs Lemma}.

\end{proof}

\begin{thm} \label{Main Theorem}
Let $L$, $T$, and  $D_0, \ldots D_{m-1}$ be as in the setup above. If any meridional disc in $T$ that misses $D=\cup_{i=0}^{i=m-1} D_i$ intersects $L$ in at least $n$ points, then  $N(L,T)=n$.
\end{thm}
\begin{proof}
Since the setup implies $N(L,T)\leq n$, it suffices to show $N(L,T)\geq n$. Let $K$ be a meridional disc of $T$ so that $\vert K\cap L\vert $ is minimal. Without loss
of generality, we may assume that   $\partial K \cap (\bigcup_{i=0}^{m-1} \partial D_i) = \emptyset$,
K is in general position with respect to D, and $L\cap K \cap D=\emptyset$. 

{Note that if 
$\partial K \cap (\bigcup_{i=0}^{m-1} \partial D_i) \neq \emptyset$, one can use standard general position techniques to make the intersection empty. For example, use the fact that there is an isotopy of $T$ taking $K$ to a close parallel copy of any particular $D_i$. This isotopy can be feathered or tapered off to the identity away from a small neighborhood of the boundary of $T$ without changing the intersection of $K$ with 
$L$.}

We show by induction on the number of components of $K\cap  D$ that $\vert K\cap L\vert\geq n$. If
$K\cap  D=\emptyset$, then by Lemma \ref{Missing M-Discs Lemma}, $\vert L \cap K\vert \geq n$. Consider a simple closed curve
component $c$ of $K\cap D$ that is innermost on K. Then $c$ bounds a disc $k$ in $K$  and a
disc $d$ in some $D_i$.  Note that $c$ may not be innermost in $D$ among the simple closed curves of $K\cap D$. However, the (possibly) singular disc $(K- k)\,\, \cup\,\, d$ meets $L$ in no more
points than $K$ by Lemma \ref{Replacement Disc Lemma}. The singularities, if any,  of $(K- k)\,\, \cup\,\, d$ consist of disjoint double curves that
can be cut apart to give a non-singular disc $K^{\prime}$ that meets L in no more points than
K does. See \cite[Ch. 4]{He04} for details on cutting apart double curves. By pushing $K^{\prime}$ slightly off $D_i$ we obtain a meridional disc $K^{\prime\prime}$ such that  $\vert K^{\prime\prime}\cap L\vert $ is still minimal and such that  $ K^{\prime\prime}\cap D$  has fewer components than $K\cap  D$. Therefore, by induction, $\vert K^{\prime\prime}\cap L\vert \geq n$ and so $\vert K\cap L\vert \geq n$.
\end{proof}

\begin{cor}
\label{Chamber Corollary}
Let $L$, $T$, and  $D_0, \ldots D_{m-1}$ and  $C_i$ be as in the setup above. If 
{$N(L\cap C_i,C_i)= n$} for each $i$, then $N(L,T)=n$.
\end{cor}
\begin{proof}
Any meridional disc that misses $\cup_{i=0}^{i=m-1} D_i$ is an interior meridional disc for some chamber, and so intersects $L$ in at least $n$ points. The result now follows from  Theorem \ref{Main Theorem}.
\end{proof}

\begin{cor}
\label{Clasp Corollary}
Let $L$, $T$, and  $D_0, \ldots D_{m-1}$ and  $C_i$ be as in the setup above. If, for each $i$, $C_i\cap L$ consists of $k_i$ clasps and $\ell_i$ spanning arcs where $2k_i+\ell=n$,  then $N(L,T)=n$.
\end{cor}
\begin{proof}
The hypotheses show that any interior meridional disc in $C_i$ intersects $L$ in at least $2k_i+\ell=n$ points. (See Lemma \ref{Chamber Index Lemma}). The result now follows from Corollary \ref{Chamber Corollary}.
\end{proof}

\begin{cor} \label{Complement Corollary}
Let $L$, $T$, and  $D_0, \ldots D_{m-1}$ and  $C_i$ be as in the setup above. Let $L^{\prime}$ be a link in the interior of a solid torus $T^{\prime}$ and let $D^{\prime}_0, \ldots D^{\prime}_{m-1}$ be a cyclically ordered collection of meridional discs in $T$, each intersecting $L^{\prime}$ transversally in $n$ points. 
Let  $C^{\prime}_i$ be the region in $T^{\prime}$ bounded by $D^{\prime}_i$ and $D^{\prime}_{(i+1)mod\ m}$.
If $N(L^{\prime},T^{\prime})=n$, and if $(C_i, C_i\cap L)\cong (C^{\prime}_i, C^{\prime}_i\cap L^{\prime})$ for each $i$, then $N(L,T)=n$.
\end{cor}
\begin{proof}
For each chamber $C^{\prime}_i$, $N(L^{\prime}\cap C^{\prime}_i, C^{\prime}_i)$ must be at least $n$. Otherwise, $N(L^{\prime}, T^{\prime})$ would be less than $n$. So for each chamber $C_i$, $N(L\cap C_i, C_i)$ must be at least $n$. The result now follows from Corollary \ref{Chamber Corollary}.
\end{proof}

\section{Examples and Applications}\label{Application Section}
We apply the theorem and corollaries from the previous section to compute the geometric index of a number of {new and old} examples to illustrate the power of these results.

{\textbf{Figure \ref{Complicated Links}} We begin with the new example(s) of Figure \ref{Complicated Links}. Divide the outer torus into eight chambers as indicated in Figure \ref{Chambers2}. Each chamber has a collection of spanning arcs, and possibly Whitehead clasps, Square knot clasps, or Antoine clasps. By Lemma \ref{Chamber Index Lemma}, the geometric index of the inner link intersected with the chamber in the chamber is 8. The result now follows from Corollary \ref{Chamber Corollary} or from Corollary \ref{Clasp Corollary}.}

\begin{figure}[ht]
\begin{center}\includegraphics[width=0.6\textwidth]{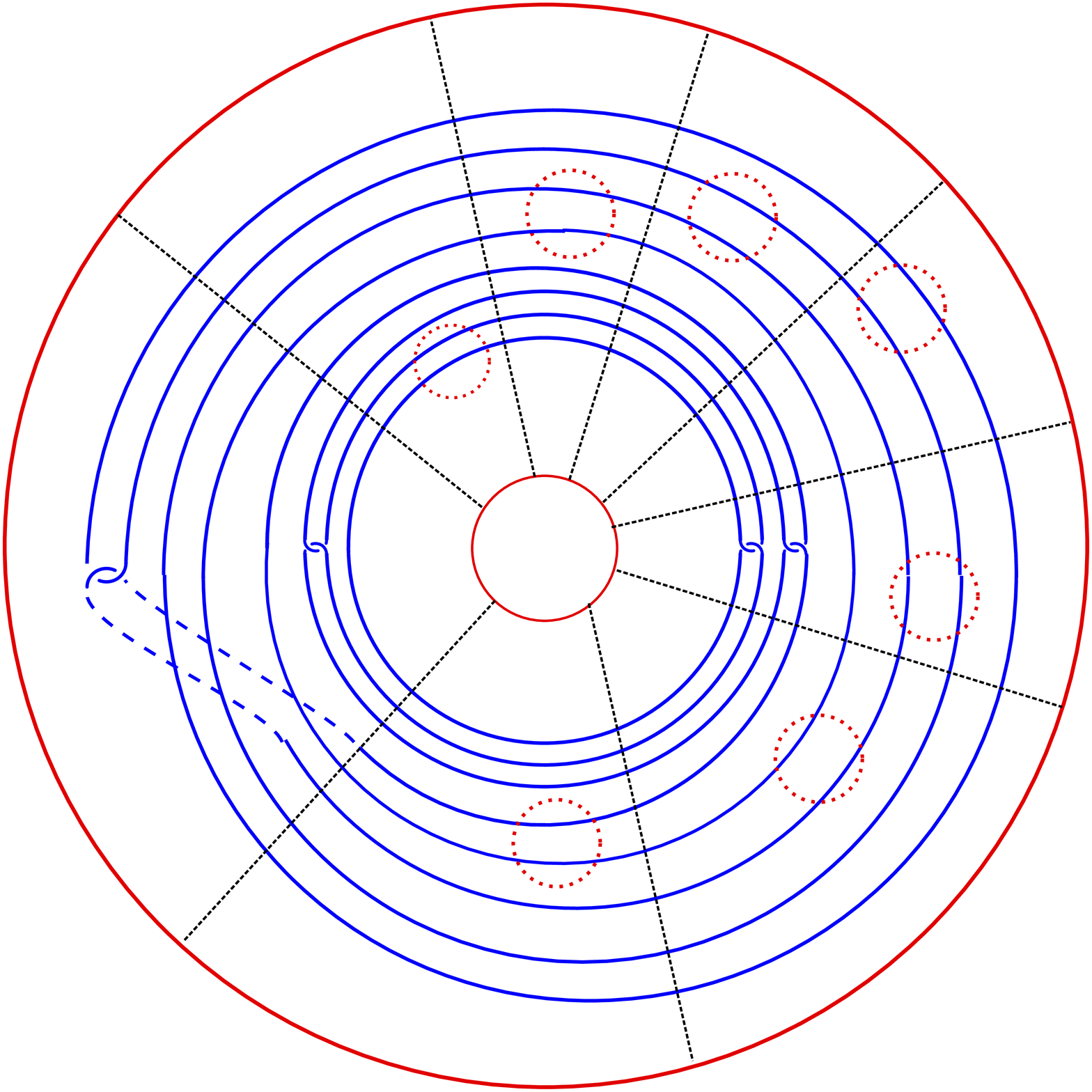}
 \end{center}
  \caption{%
    Chambers for Figure  \ref{Complicated Links} }%
  \label{Chambers2}
\end{figure}

\begin{figure}[hbt]
     \begin{center}
        \subfigure[Antoine Link with discs]{%
            \label{Antoine-D}
            \includegraphics[width=0.31\textwidth]{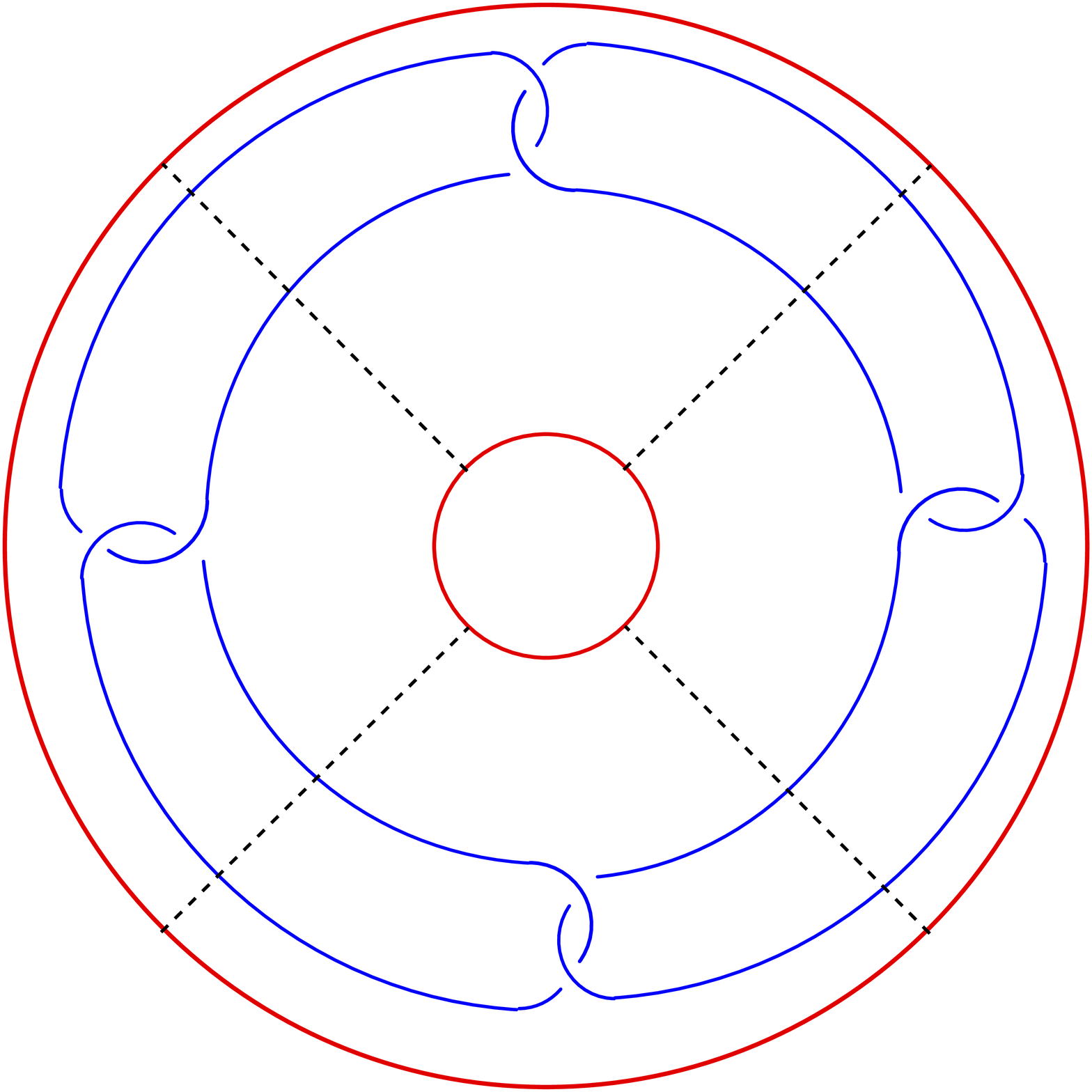}
          }
          \subfigure[Knotted Link with discs]{%
            \label{Knotted-D}
            \includegraphics[width=0.31\textwidth]{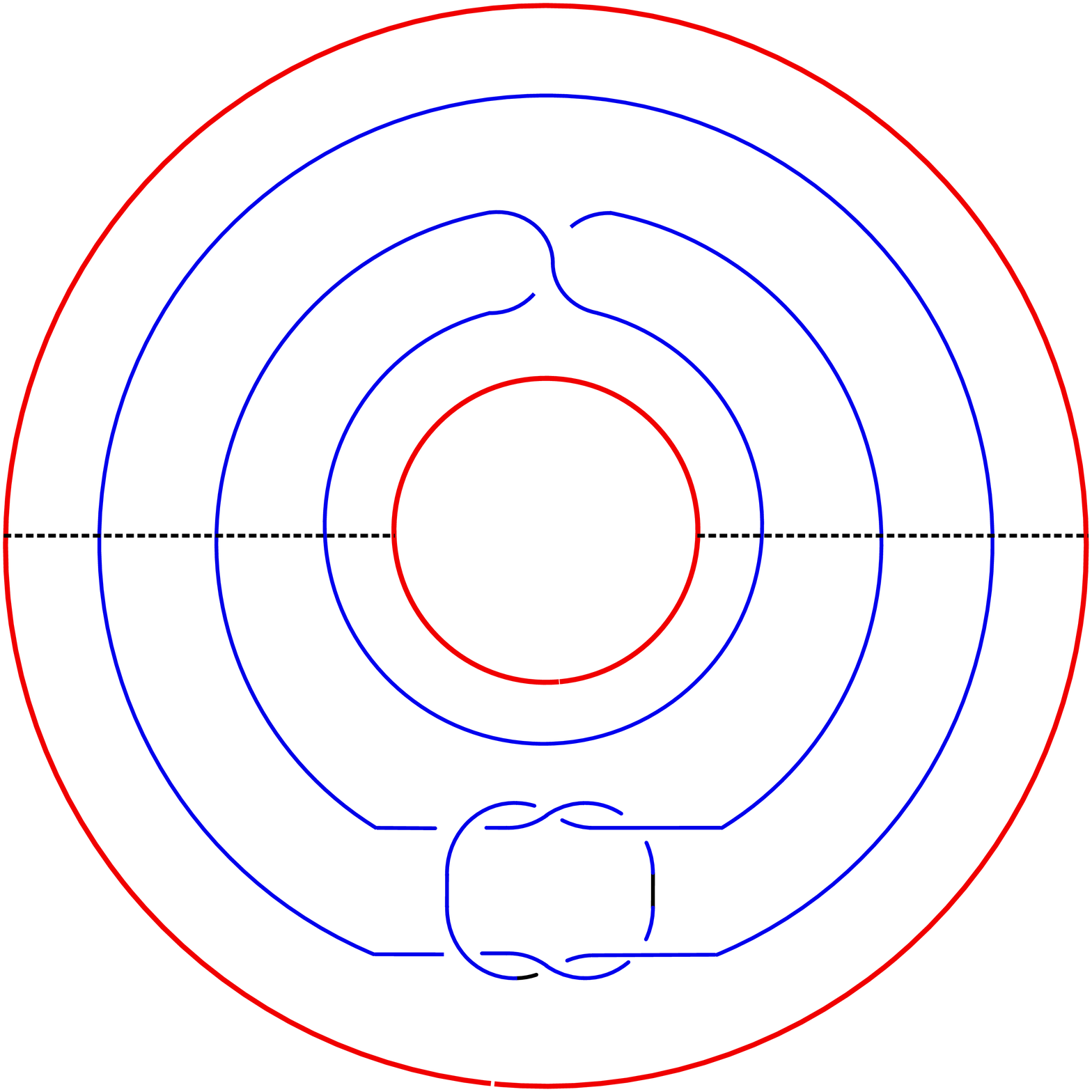}
            }
            \subfigure[McMillan Link with discs]{%
            \label{McM-D}
            \includegraphics[width=0.31\textwidth]{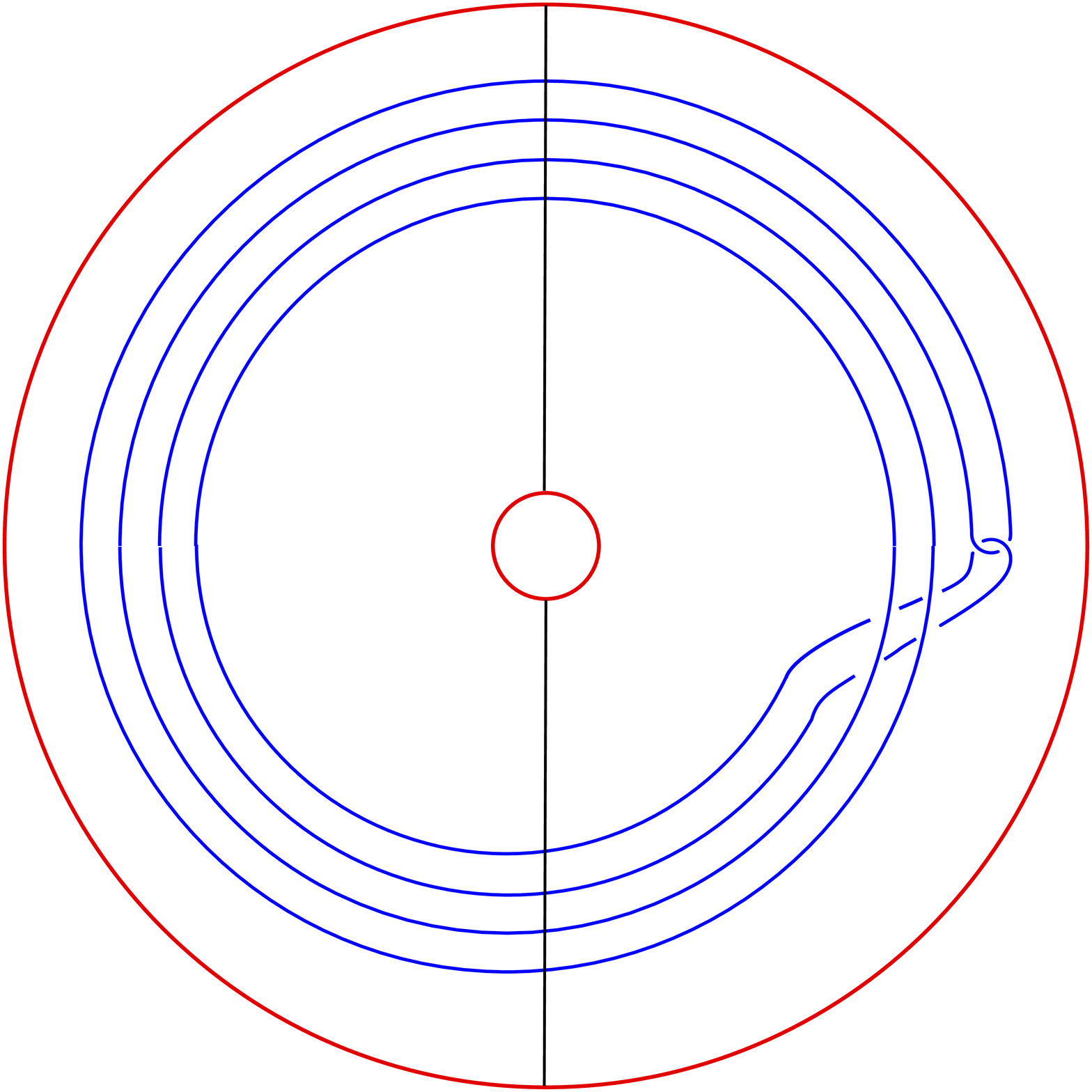}
            }\\
             \subfigure[Shifted Link with discs]{%
            \label{Gabai Shifted}
            \includegraphics[width=0.31\textwidth]{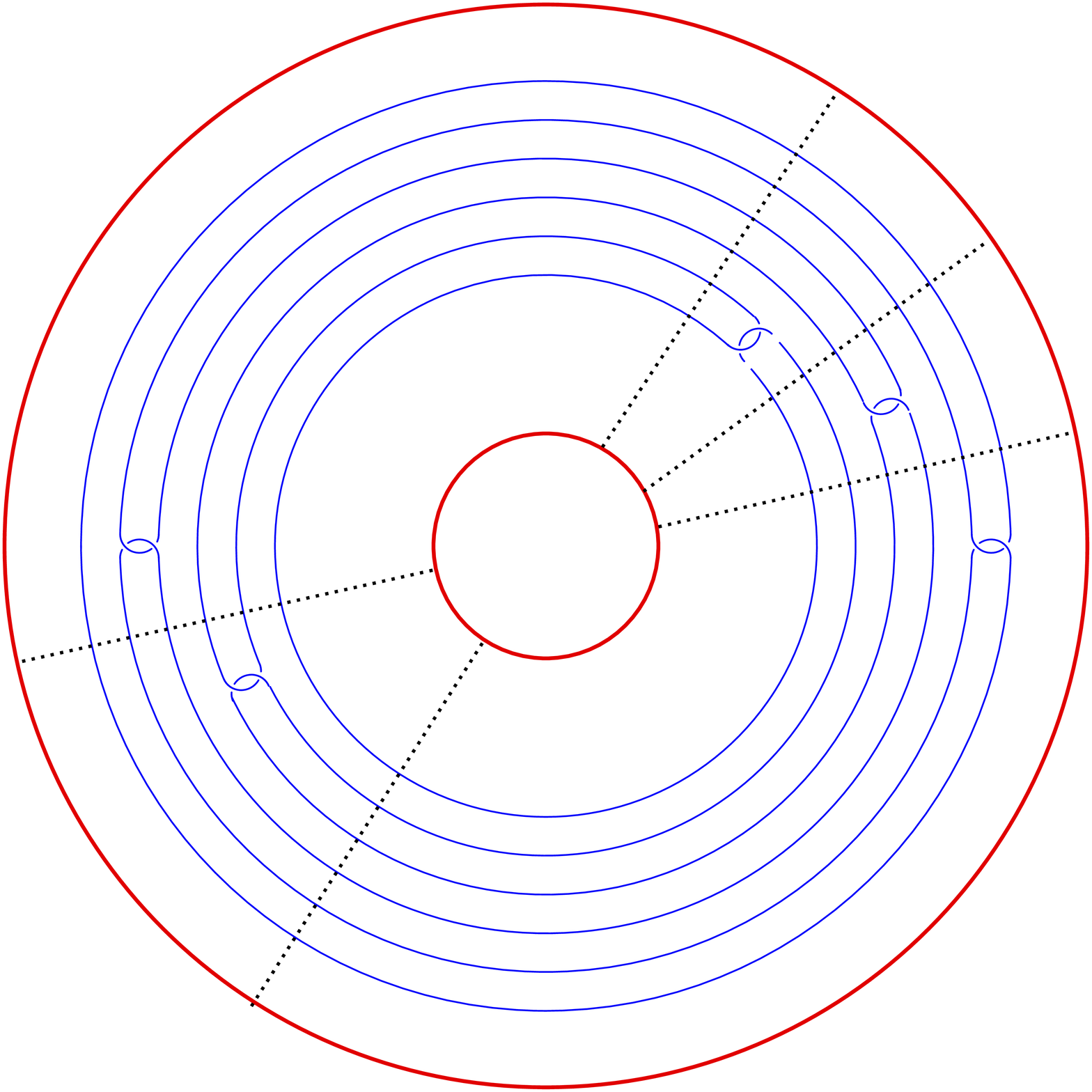}
          }\subfigure[Gabai Link with discs]{%
            \label{Gabai-D}
            \includegraphics[width=0.31\textwidth]{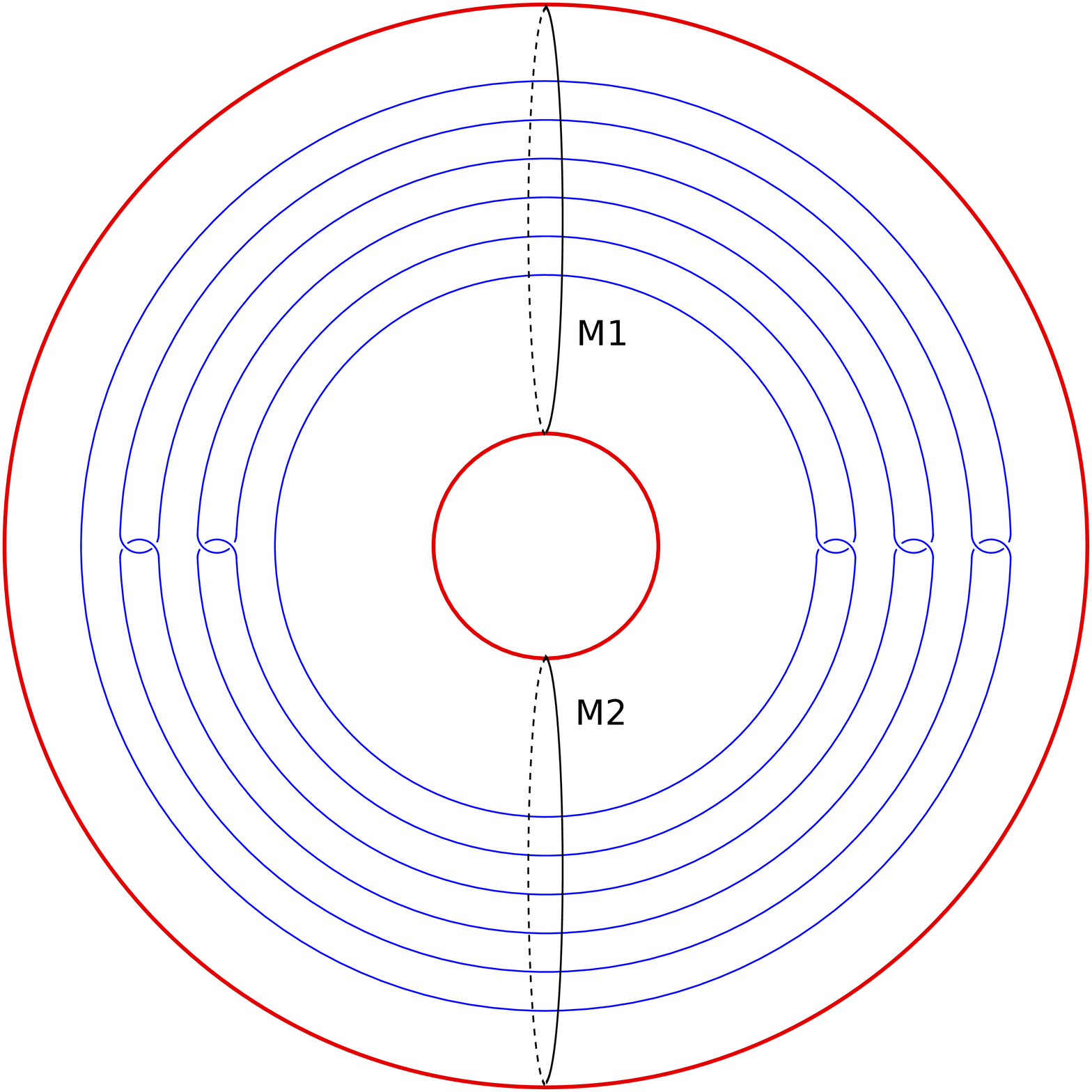}
          }
          \subfigure[Bing and Whitehead Links with \hspace*{7em}discs]{%
            \label{BW-D}
            \includegraphics[width=0.31\textwidth]{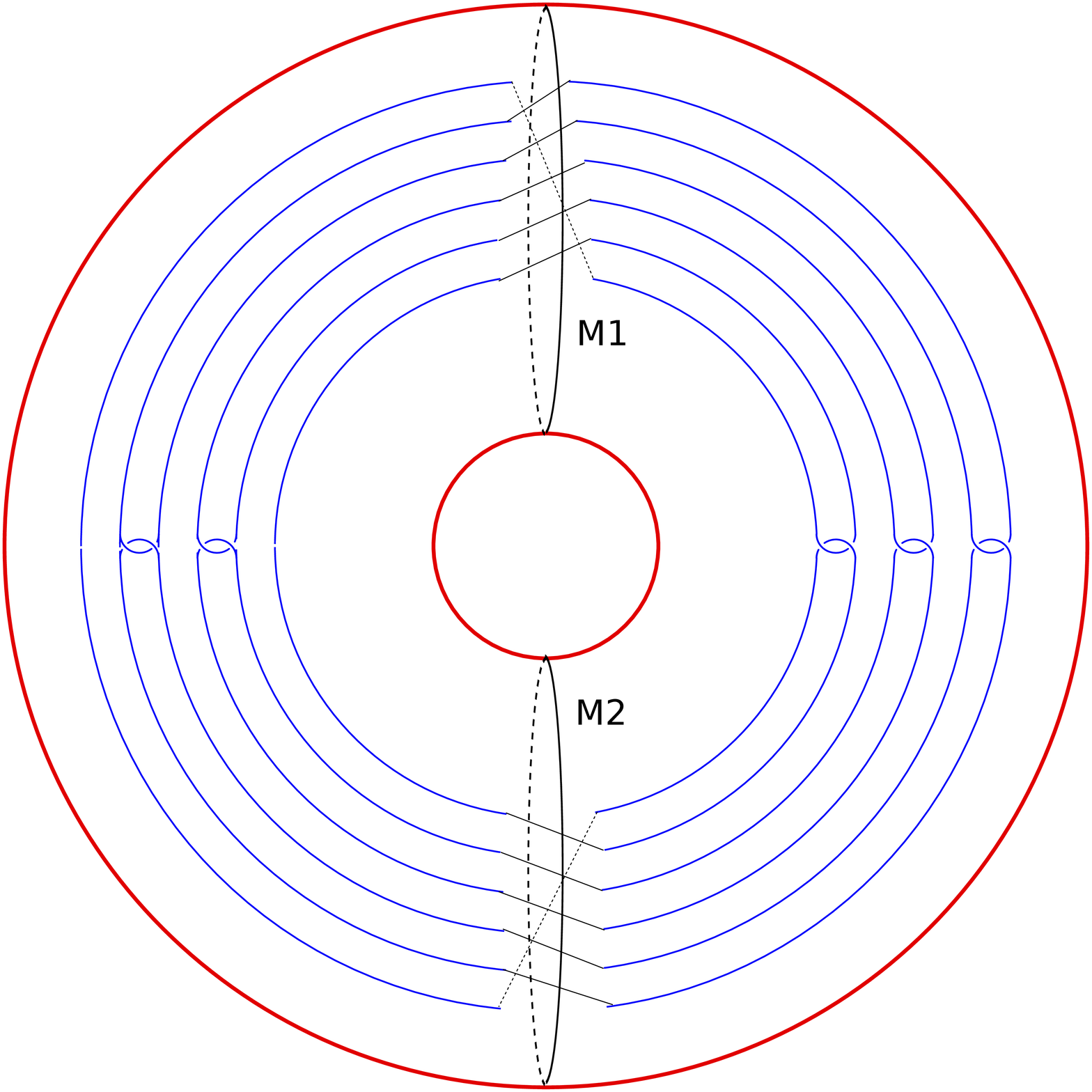}
        }
    \end{center}
    \caption{%
        Links with Added Meridional Discs }%
   \label{G-index}
\end{figure}

\textbf{Figures \ref{Whitehead}, \ref{Bing}, and \ref{Antoine}:} By Lemma 2.5, the geometric index of the link in each of these figures on the outer torus is 0 or 2. So it suffices to show the geometric index is not 0. The outer torus and link in Figure \ref{Antoine} is a covering space of the outer tori and links in Figures \ref{Whitehead} and \ref{Bing}. 
Any meridional disc in the outer torus in Figures \ref{Whitehead} and \ref{Bing} that misses the interior link lifts to a meridional disc in the outer torus in Figure \ref{Antoine} that misses the interior link. So it suffices to show the geometric index of the link in Figure \ref{Antoine} is nonzero.
There is a standard cut and paste argument that shows this. See for example \cite[Section 9]{Dav07}
or \cite[Chapter 3-G]{Rol90}.

An alternate approach using Corollary \ref{Clasp Corollary} is as follows. Divide the outer torus in four chambers as indicated in Figure \ref{Antoine-D}. Each chamber has a single Whitehead clasp in it. By Lemma \ref{Chamber Index Lemma}, the geometric index of the inner link intersected with the chamber in the chamber is 2. The result now follows from Corollary \ref{Chamber Corollary}.

\textbf{Figure \ref{Algebraic2}:} The algebraic index is 2 since the inner curve has winding number 2 in the outer torus. Since the geometric index is greater than or equal to the algebraic index, and since there is clearly a meridional disc that intersects the inner curve twice, the geometric index is 2.

\textbf{Figure \ref{McM4}:} Divide the outer torus into two chambers as indicated in Figure \ref{McM-D}. Each chamber has a Whitehead clasp and two spanning arcs or four spanning arcs. By Lemma \ref{Chamber Index Lemma}, the geometric index of the inner link intersected with the chamber in the chamber is 4. The result now follows from Corollary \ref{Chamber Corollary} or from Corollary \ref{Clasp Corollary}.

\textbf{Figure \ref{Knotted}:} Divide the outer torus into two chambers as indicated in Figure \ref{Knotted-D}. Each chamber has a Whitehead clasp and spanning arc or a Square Knot clasp and spanning arc. By Lemma \ref{Chamber Index Lemma}, the geometric index of the inner link intersected with the chamber in the chamber is 3. The result now follows from Corollary \ref{Chamber Corollary} or from Corollary \ref{Clasp Corollary}.

\textbf{Figure  \ref{Gabai6}:} In Figure \ref{Gabai Shifted}, we have slightly shifted the inner link by a homeomorphism and divided the outer torus into 5 chambers. The inner link, intersected with each chamber, consists of a Whitehead clasp and four spanning arcs. By Lemma \ref{Chamber Index Lemma}, the geometric index of the inner link intersected with the chamber in the chamber is 6. The result now follows from Corollary \ref{Chamber Corollary} or from Corollary \ref{Clasp Corollary}.

Alternately, consider Figures \ref{Gabai-D} and \ref{BW-D}. Corollary \ref{Complement Corollary} applies directly since the complements of the two indicated meridional discs are homeomorphic. Figure \ref{BW-D} consists of two Bing links and a Whitehead link, so the geometric index of the inner link in the outer torus is 6. By Corollary \ref{Complement Corollary}, the geometric index of the inner link in the outer torus in Figure \ref{Gabai-D} is also 6.

\section{Acknowledgements}

{The authors would like to thank the referee for helpful suggestions and for pointing out useful clarifications to include.} The authors were supported in part by the Slovenian Research Agency grant BI-US/15-16-029. 
The second author was supported in part by the National Science Foundation grant DMS0453304.
The second and fourth authors were supported in part by the National Science Foundation grant DMS0707489. 
The third author was supported in part by the Slovenian Research Agency grants P1-0292 and J1-6721.

\end{document}